\documentclass{article}

\usepackage{geometry, color}
\usepackage[english]{babel}
\usepackage{amsmath,amsthm,amssymb}
\usepackage{dsfont}
\usepackage{picture,float,graphicx,listings}
\usepackage{hyperref}

\newtheorem{theorem}{Theorem}[subsection]
\newtheorem{lemma}[theorem]{Lemma}

\newtheorem{corollary}[theorem]{Corollary}
\newtheorem{proposition}[theorem]{Proposition}
\theoremstyle{remark}
\newtheorem*{remark}{Remark}
\newtheorem*{conjecture}{Conjecture}
\theoremstyle{definition}
\newtheorem{definition}[theorem]{Definition}
\newtheorem{example}[theorem]{Example}

\newtheorem{appendix theorem}{Theorem}[section]
\newtheorem{appendix lemma}[appendix theorem]{Lemma}
\newtheorem{appendix corollary}[appendix theorem]{Corollary}
\newtheorem{appendix proposition}[appendix theorem]{Proposition}

\newcommand{\N}{\mathbb{N}}
\newcommand{\Z}{\mathbb{Z}}
\newcommand{\Q}{\mathbb{Q}}
\newcommand{\R}{\mathbb{R}}
\newcommand{\C}{\mathbb{C}}

\renewcommand{\S}{\mathbb{S}}
\newcommand{\T}{\mathbb{T}}

\newcommand{\V}{V}
\newcommand{\A}{\mathcal{A}}
\newcommand{\sA}{\mathrm{s}\mathcal{A}}
\newcommand{\U}{\textup{U}}

\newcommand{\tend}[2]{\underset{#1 \rightarrow #2}{\longrightarrow}}
\newcommand{\landau}[2]{\underset{#1 \rightarrow #2}{=}}
\newcommand{\eq}[2]{\underset{#1 \to #2}{\sim}}
\newcommand{\ind}{\mathds{1}}

\newcommand{\Mat}{\mathcal{M}\mathrm{at}}
\newcommand{\Span}{\mathrm{Span}}

\newcommand{\Img}{\mathrm{Im}}

\newcommand{\GL}{\mathrm{GL}}

\newcommand{\tr}{\mathrm{tr}}
\renewcommand{\det}{\mathrm{det}}
\newcommand{\diag}{\mathrm{diag}}

\newcommand{\Vol}{\mathrm{Vol}}

\newcommand{\diff}{\mathrm{d}}
\renewcommand{\i}{\mathrm{\textbf{i}}}
\newcommand{\e}{\mathrm{\textbf{e}}}

\newcommand{\Log}{\mathrm{Log}}
\newcommand{\sLog}{\mathrm{sLog}}
\newcommand{\trop}{\mathrm{trop}}
\newcommand{\strop}{\mathrm{strop}}
\newcommand{\Conv}{\mathrm{Conv}}
\newcommand{\Newt}{\mathrm{Newt}}
\newcommand{\sNewt}{\mathrm{sNewt}}
\newcommand{\val}{\mathrm{val}}
\newcommand{\sval}{\mathrm{sval}}
\newcommand{\RR}{\mathcal{R}}

\newcommand{\abs}[1]{\left| #1 \right|}
\newcommand{\norme}[1]{\left\| #1 \right\|}

\renewcommand{\d}{\displaystyle}
\renewcommand{\leq}{\leqslant}
\renewcommand{\geq}{\geqslant}
\newcommand{\inter}[1]{\overset{\circ}{#1}}

\newcommand{\fonction}[4]{\left\{ \begin{array}{rcl}
		\d #1 & \rightarrow & \d #2\\
		\d #3 & \mapsto & \d #4
	\end{array} \right.}

\title{Generalized amoebas for subvarieties of $\GL_n(\C)$}
\author{Rémi Delloque\footnote{École Normale Supérieure de Lyon, Lyon, France, e-mail: remi.delloque@ens-lyon.fr},
Kiumars Kaveh\footnote{University of Pittsburg, Pittsburgh, PA, U.S.A., e-mail: kaveh@pitt.edu}}

\begin{document}

\maketitle

\begin{abstract}
This paper is a report based on the results obtained during a three months internship at the University of Pittsburgh by the first author and under the mentorship of the second author. In \cite{Kaveh} and \cite[Section 7]{Kaveh-Manon}, the notion of an amoeba of a subvariety in a torus $(\C^*)^n$ has been extended to subvarieties of the general linear group $\GL_n(\C)$. 
In this paper, we show some basic properties of these matrix amoebas, e.g. any such amoeba is closed and the connected components of its complement are convex when the variety is a hypersurface. We also extend the notion of Ronkin function to this setting. For hypersurfaces, we show how to describe the asymptotic directions of the matrix amoebas using a notion of Newton polytope. Finally, we partially extend the classical statement that the amoebas converge to the tropical variety. We also discuss a few examples. Our matrix amoeba should be considered as the Archimedean version of the spherical tropicalization of Tevelev-Vogiannou for the variety $\GL_n(\C)$ regarded as a spherical homogeneous space for the left-right action of $\GL_n(\C) \times \GL_n(\C)$. \end{abstract}

\medskip
{\bf This is a preliminary version, comments are welcome.}

\tableofcontents

\section*{Introduction}
In \cite[Section 7]{Kaveh-Manon} and \cite{Kaveh}, the logarithm of singular values of a matrix has been suggested as an analogue of the logarithm map on the algebraic torus $(\C^*)^n$ for the general linear group $\GL_n(\C)$. In this paper we establish some basic results about the image of subvarieties in $\GL_n(\C)$ under this logarithm map, extending the classic results about amoebas in $(\C^*)^n$. 

We start with some background and motivations. From the point of view of algebraic geometry, tropical geometry is concerned with describing the ``(exponential) behavior at infinity'', of subvarieties in $(\C^*)^n$ where $\C^* = \C\backslash\{0\}$. With componentwise multiplication, $(\C^*)^n$ is an abelian group. It is usually referred to as an \emph{algebraic torus} and is one of the basic examples of algebraic groups. A subvariety of $(\C^*)^n$ is called 
a {\it very affine variety}. The behavior at infinity of a subvariety $Y \subset (\C^*)^n$ is encoded in a union of convex polyhedral cones called the {\it tropical variety} of $Y$. 
There are (at least) two natural ways to define the tropical variety of a very affine variety: (1) using the formal Laurent series and \emph{tropicalization map}, and (2) using the \emph{logarithm map}. 

\textit{Tropicalization map (on torus):} Let $\mathcal{K} = \C((t))$ be the field of formal Laurent series in one indeterminate $t$. 
Then the algebraic closure $\overline{\mathcal{K}} = \C\{\!\{t\}\!\} =  \bigcup_{k = 1}^\infty \C((t^{1/k}))$ is the field of formal Puiseux series. 
The field $\overline{\mathcal{K}}$ comes equipped with the {\it order of vanishing} valuation 
$\val: (\overline{\mathcal{K}})^* = \overline{\mathcal{K}}\backslash\{0\} \to \Q$ defined as follows: for a Puiseux series $f(t) = \sum_{i = m}^\infty a_i t^{i/k}$, where $a_m \neq 0$, we put $\val(f(t)) = m/k$.
The valuation $\val$ gives rise to the {\it tropicalization map} $\trop$ from $(\overline{\mathcal{K}}^*)^n$ to $\Q^n$:
\[
\trop(z_1(t), \ldots, z_n(t)) = (\val(z_1(t)), \ldots, \val(z_n(t))).
\]
Let $Y \subset (\C^*)^n$ be a subvariety with ideal $I=I(Y) \subset \C[x_1^\pm, \ldots, x_n^\pm]$. Let $Y(\overline{\mathcal{K}})$ denote the Puiseux series valued points on $Y$, that is, $Y(\overline{\mathcal{K}}) = \{z(t) = (z_1(t), \ldots, z_n(t)) \in (\overline{\mathcal{K}}^*)^n \mid \forall f \in I,~f(z_1(t), \ldots, z_n(t)) = 0\}$. The tropical variety of $Y$ is the closure (in $\R^n$) of the image of $Y(\overline{\mathcal{K}})$ under the map $\trop$. One shows that the tropical variety of a subvariety always has the structure of a \emph{fan} in $\R^n$, that is, it is a finite union of (strictly) convex rational polyhedral cones (see \cite[Chapter 3]{MacLagan}).
 
\textit{Logarithm map (on torus):}
The {\it logarithm map} $\Log: (\C^*)^n \to \R^n$ is defined by:
\begin{equation} \label{equ-log-map-torus}
\Log(z_1, \ldots, z_n) = (\ln|z_1|, \ldots, \ln|z_n|).
\end{equation}
Clearly the inverse image of every point is an $(S^1)^n$-orbit in $(\C^*)^n$. Here $S^1$ denotes the complex unit circle and $(S^1)^n = \{(z_1, \ldots, z_n) \mid |z_1| = \cdots = |z_n| = 1\}$ which is the maximal compact subgroup in $(\C^*)^n$. 

For a subvariety $Y \subset (\C^*)^n$, its \emph{(Archimedean) amoeba} $\mathcal{A}(Y)$ is the image of $Y$ in $\R^n$ under the logarithm map $\Log$. Amoebas were introduced by Gelfand, Kapranov and Zelevinsky in \cite[Section 6.1]{Gelfand}, as a means to study the asymptotic behavior at infinity of subvarieties in $(\C^*)^n$. An amoeba goes to infinity along certain directions usually called its {\it tentacles} (and hence the name amoeba). The directions along which an amoeba goes to infinity in fact coincides with the tropical variety of $Y$. More precisely, we have the following fact that goes back to Bergman \cite{Bergman} (in a different form and before the notion of tropical variety was introduced):
\begin{center}
\emph{As $\rho \to 0^+$, the rescaled amoeba $\rho \mathcal{A}(Y)$ approaches $-\trop(Y)$, the negative of the tropical variety.}
\end{center}
When $Y$ is a hypersurface this is relatively easy to show and basically appears in \cite[Section 6.1, Proposition 1.9]{Gelfand}. Even though the statement that, for arbitrary $Y$, the amoeba approaches the tropical variety has been known as a folklore, a precise formulation and proof only appeared relatively recently in (\cite[Theorem A]{Jonsson}). 

It is natural to ask whether tropical geometry and notions of tropicalization and logarithm map can be extended to other classes of varieties with group actions. To this end, it is natural to consider \emph{spherical homogeneous spaces} $G/H$ where $G$ is a reductive algebraic group over $\C$. We recall that a $G$-variety is called \emph{spherical} if a Borel subgroup (and hence all Borel subgroups) have an open (hence dense) orbit. The notion of tropicalization has been extended to spherical homogeneous spaces in the work of Tevelev and Vogiannou \cite{Tevelev}. A suggestion for the notion of logarithm map on spherical homogeneous spaces appears in \cite{Kaveh-Manon}. In the case where the homogeneous space is $\GL_n(\C)$ this logarithm map coincides with the logarithm of singular values of a matrix \cite{Kaveh}. Here we consider $\GL_n(\C)$ as a spherical homogeneous space for the left-right action of $G = \GL_n(\C) \times \GL_n(\C)$, thus identifying $\GL_n(\C)$ with $(\GL_n(\C) \times \GL_n(\C)) / \GL_n(\C)_\textup{diag}$ where $\GL_n(\C)_\textup{diag} = \{(g, g) \mid g \in \GL_n(\C)\}$.

\textit{Main results:}
For an $n \times n$ matrix $A$ we let $\sLog(A)$ to be the collection of logarithms of singular values of $A$ (see Definition \ref{def-sLog}). This defines the \emph{spherical logarithm} map $\sLog: \GL_n(\C) \to \R^n / \mathcal{S}_n$, where $\mathcal{S}_n$ is the group of permutations (symmetric group). We call the image of a subvariety $Y \subset \GL_n(\C)$ under the logarithm map $\sLog$, the \emph{matrix amoeba} or \emph{spherical amoeba} of $Y$ and denote it by $\sA(Y)$.

In this paper, for any subvariety $Y \subset \GL_n(\C)$, we show the following:
\begin{itemize}
\item The matrix amoeba $\sA(Y)$ is closed.
\item Each connected component of the complement of $\sA(Y)$ is convex when $Y$ is a hypersurface.
\item We give an analogue of the notion of Ronkin function and show that is is affine on each connected component of the complement of $\sA(Y)$.
\item For a regular function $f \in \C[\GL_n]$ we consider its \emph{spherical Newton polytope} (also called its \emph{weight polytope}). When $Y$ is a hypersurface given by an equation $f = 0$, we give a description of the asymptotic directions in $\sA(Y)$ (in other words, the spherical tropical variety of $Y$) in terms of the spherical Newton polytope of $f$.
\item We show that the limit of the sets $\rho\, \sA(Y)$, as $\rho \to 0^+$, contains the spherical tropical variety of $Y$ (in the sense of Tevelev-Vogiannou). Moreover, it coincides with the spherical tropical variety when $Y$ is a hypersurface.
\end{itemize}

\begin{remark}
We point out that a very general construction of a logarithm map and amoeba appears in \cite{Eliyashev}. It is an interesting question to investigate the connection between our notion of logarithm map and the one in \cite{Eliyashev} (in the case of general linear group).
\end{remark}

\begin{remark}
We expect that the constructions, statements and proofs in the present note, with little change, extend to arbitrary connected reductive algebraic groups over $\C$. 
\end{remark}


Section \ref{SEC:Definitions} contains the definitions of matrix amoebas and some basic properties that will justify the definitions. Section \ref{SEC:Amoebas} is a study of the geometric aspect of matrix amoebas of hypersurfaces. Section \ref{SEC:Representations} is a small digression about representations of the general linear group and Newton polytopes. Section \ref{SEC:Tropical} generalises the Bergman theorem \cite{Bergman} that links amoebas to tropical varieties. Finally, appendices \ref{SEC:Notations} and \ref{SEC:Lemmas} are dedicated respectively to notations and technical lemmas that are not linked to tropical geometry or amoebas.

\section{Definitions and elementary results}\label{SEC:Definitions}

\subsection{Definitions and results in the torus}
We recall that the algebraic torus $T = (\C^*)^n$ and the matrix group $\GL_n(\C)$ are affine varieties:
\[
T \cong \left\{(z_1,\ldots,z_n,w) \in \C^{n + 1}|z_1\cdots z_nw = 1\right\},
\]
\[\GL_n(\C) \cong \left\{((a_{ij})_{1 \leq i,j \leq n},z) \in \C^{n^2 + 1}|\det(a_{ij})z = 1\right\}.
\]
A subvariety of the algebraic torus is the set of all the common zeros of the functions of an ideal $I$ of $\C[T]$, the ring of regular functions on $T$. In fact, $\C[T] = \C[X^\pm] = \C[X_1^\pm,\ldots,X_n^\pm]$ is the ring of Laurent polynomials with $n$ indeterminates. The concept of amoeba of a very affine variety, that is, a subvariety of $T$, was introduced by Gelfand, Kapranov and Zelevinsky in \cite{Gelfand}. The amoeba of a very affine variety $Y = \V(I)$ is defined to be its image under the $\Log$ map,
\[
\Log : \fonction{T}{\R^n}{z}{(\ln\abs{z_1},\ldots,\ln\abs{z_n}).}
\]
We denote it by $\A(Y)$ or $\A(I)$, or $\A(f)$ when $I = (f)$ is principal. It is known that an amoeba is a closed subset of $\R^n$ and all the connected components of its complement are convex. The asymptotic directions along which an amoeba approaches infinity is a finite union of polyhedral cones which is the \emph{tropical variety of $Y$}. When $I = (f)$ is a principal ideal of $\C[X^\pm]$, one can describe the tropical variety using the Newton polytope of $f$: The \emph{support} $S_f$ of $f$ is the set of exponents $m=(m_1, \ldots, m_n) \in \Z^n$ such that the coefficient of $f$ at $X^m = X_1^{m_1}\cdots X_n^{m_n}$ is non zero and the \emph{Newton polytope} $\Newt(f)$ of $f$ is the convex hull of its support. Bergman showed in \cite{Bergman} that the asymptotic directions on which $\A(f)$ goes to infinity (that is, its tropical variety), coincides with the $(n-1)$-skeleton of the normal fan of $\Newt(f)$. For more details, see \cite[Section 1.4]{MacLagan}.

\subsection{Definitions in $\GL_n(\C)$}
To extend the notion of amoeba to other classes of varieties (in place of the torus) we need an extension of the notion of $\Log$ map. 
We note that the logarithm map $\Log: (\C^*)^n \to \R^n$ is invariant by multiplication with elements of the compact torus $(S^1)^n$. The compact torus is the maximal compact subgroup of $(\C^*)^n$. Similarly, the unitary group $U(n)$ is a maximal compact subgroup of $\GL_n(\C)$. Recall from linear algebra that the \emph{singular values decomposition} states that $GL_n(\C) = U(n) D_n U(n)$ where $A_n$ is the subgroup of diagonal matrices with positive real entries. If we write $A \in \GL_n(\C)$ as $A = P D Q$ where $P$, $Q$ are unitary matrices and $D$ is diagonal with positive diagonal entries, the diagonal entries of $D$ are the \emph{singular values} of $A$.   
\begin{definition}  \label{def-sLog}
Following \cite{Kaveh-Manon} we define the \emph{matrix logarithm map} (or \emph{spherical logarithm map}) on $\GL_n(\C)$ as follows:
    \[
    \sLog : \fonction{\GL_n(\C)}{\R^n/\mathcal{S}_n}{A}{(\ln(\lambda_1),\ldots,\ln(\lambda_n)) \textrm{ where the $\lambda_k$ are the singular values of $A$.}}
    \]
Here $\mathcal{S}_n$ is the symmetric group (the group of permutation of $\{1, \ldots, n\}$) which acts on $\R^n$ by permuting the coordinates. 
\end{definition}
We remark that the \emph{s} in $\sLog$ stands for \emph{spherical}. This is because, $\GL_n(\C)$ with left-right action of $\GL_n(\C) \times \GL_n(\C)$ is an important example of a spherical homogeneous space.  

By abuse of terminology and notation, we may identify subsets of $\R^n / \mathcal{S}_n$ with subsets of $\R^n$ invariant under permutations and $\R^n/\mathcal{S}_n$. Similarly, we identify functions on $\R^n/\mathcal{S}_n$ with functions on $\R^n$ that are invariant under permutations of the coordinates.

We note that $\abs{\sLog(A)} \rightarrow +\infty$ when at least one of the entires of $A$ approaches infinity or when $A$ approaches a non-invertible matrix. We will see in Section \ref{SEC:Tropical} other reasons that make of $\sLog$ a good generalisation of $\Log$. Recall that the singular values of a matrix $A$ are the square roots of the eigenvalues of the hermitian non negative matrix (positive when $A$ is invertible) $AA^*$ (or $A^*A$). Any matrix $A$ can be written as $UDV^*$ with $U$ and $V$ invertible and $D$ non negative diagonal. Then, the diagonal coefficients of $D$ are the singular values of $A$.
\begin{definition}
    When $I \subset \C[\GL_n]$ (the ring of regular functions on $\GL_n(\C)$) is an ideal, the matrix amoeba of a matrix spherical variety $Y = \V(I)$ is $\sA(Y) = \sA(I) = \sLog(Y) \subset \R^n/\mathcal{S}_n$. We shall refer to $\sA(I)$ as \emph{matrix amoeba} (or \emph{spherical amoeba}) of $Y$. When there is no ambiguity we may simply refer to it as the amoeba of $Y$.
\end{definition}
We will need a last definition that will help us to make the link between classical and matrix amoebas.
\begin{definition}
Let $f \in \C[\GL_n(\C)]$ be a regular function and let $A$, $B$ be invertible matrices. We define $\Psi_{A, B}$ as follows:
    \[\Psi_{A,B}(f) : \fonction{T}{\C}{z}{f(A\,\diag(z)\,B^{-1})},
    \]
where $\diag(z)$ is the diagonal matrix with coordinates of $z$ as diagonal entries. $\Psi_{A,B}$ is a $\C$-algebra homomorphism from $\C[\GL_n(\C)]$ to $\C[X^\pm]$.
\end{definition}
We recall that $\C[\GL_n(\C)]$ is the ring of functions of the form $\det^Nf$ where $f$ is a polynomial in the matrix entries and $N \in \Z$. 

\subsection{Elementary properties of matrix amoebas}
We begin by showing that matrix amoebas are closed in $\R^n/\mathcal{S}_n$ (with respect to the natural topology on it). Since all subvarieties are closed, it is enough to show that $\sLog$ is a closed map.
\begin{lemma}\label{LEM:Norm sLog}
For any $A \in \GL_n(\C)$ we have:
	\[
	\abs{\sLog(A)}_\infty = \frac{1}{2}\max\left\{\ln\norme{AA^*},\ln\norme{(AA^*)^{-1}}\right\} = \max\left\{\ln\norme{A},\ln\norme{A^{-1}}\right\} + \mathrm{O}(1)
	\]
	where $\abs{\cdot}_\infty$ is the infinite norm and $\norme{\cdot}$ is the operator norm associated the the Euclidean norm on $\C^n$.
\end{lemma}
\begin{proof}
Let $A \in \GL_n(\C)$. The matrix $AA^*$ is a positive hermitian matrix so it can be written as $U\,\diag(\lambda)\,U^*$ with $U$ unitary and $\lambda = (\lambda_1 \geq \cdots \geq \lambda_n > 0)$. Moreover, being hermitian, its norm is equal to its spectral radius $\lambda_1$. Also $(AA^*)^{-1} = U\,\diag(\lambda_1^{-1},\ldots,\lambda_n^{-1})\,U^*$. Thus the norm of $(AA^*)^{-1}$ is $\lambda_n^{-1}$ and we have $\d \frac{1}{2}\max\left\{\ln\norme{AA^*},\ln\norme{(AA^*)^{-1}}\right\} = \frac{1}{2}\max\left\{\ln(\lambda_1),-\ln(\lambda_n)\right\}$. On the other hand, the singular values of $A$ are the square roots of the eigenvalues of $AA^*$, namely,  $\sqrt{\lambda_1},\ldots,\sqrt{\lambda_n}$. We deduce that
\begin{align*}
	\abs{\sLog(A)}_\infty & = \max\left\{\abs{\ln\left(\sqrt{\lambda_1}\right)},\ldots,\abs{\ln\left(\sqrt{\lambda_n}\right)}\right\}\\
	& = \frac{1}{2}\max\left\{\ln(\lambda_1),-\ln(\lambda_1),\ldots,\ln(\lambda_n),-\ln(\lambda_n)\right\}\\
	& = \frac{1}{2}\max\left\{\ln(\lambda_1),-\ln(\lambda_n)\right\},
\end{align*}
which proves the first equality.

For a nonzero matrix $A \in \Mat_n(\C)$, let $C_A = \frac{\norme{AA^*}}{\norme{A}^2}$. Since $\norme{AA^*} \leq \norme{A}\norme{A^*} = \norme{A}^2$ we have $0 < C_A \leq 1$. Let $C = \min_{\norme{A} = 1} C(A)$. This $\min$ is well-defined because the unit sphere is compact and $A \mapsto C_A$ is continuous. We conclude that for all nonzero matrices $A$ we have:
\[
C_A = \frac{\norme{AA^*}}{\norme{A}^2} = \norme{\frac{A}{\norme{A}}\left(\frac{A}{\norme{A}}\right)^*} \geq C > 0,
\]
because $\d \frac{A}{\norme{A}}$ is in the unit sphere. We deduce that $A \mapsto \ln(C_A) = \mathrm{O}(1)$, so
\[
\frac{1}{2}\max\{\ln\norme{AA^*},\ln\norme{(AA^*)^{-1}}\} = \max\{\ln\norme{A},\ln\norme{A^{-1}}\} + \mathrm{O}(1),
\]
as required.
\end{proof}
\begin{proposition}\label{PRO:sLog continuous}
    $\sLog$ is continuous for the distance $\d d : (x,y) \mapsto \min_{\sigma \in \mathcal{S}_n}\{\sigma \cdot x - y\}$ on $\R^n/\mathcal{S}_n$.
\end{proposition}
\begin{proof}
First of all, it is straightforward to see that $d$ is a well-defined distance for which
\[
\varphi : \fonction{\{x \in \R^n|x_1 \leq \cdots \leq x_n\}}{\R^n/\mathcal{S}_n}{x}{\mathcal{S}_n \cdot x}
\]
is an isometry. If $(A_m)_{m \in \N}$ is a sequence of invertible matrices that converges to some invertible matrix $A$. Then $A_mA_m^* \tend{m}{+\infty} AA^*$ and thus the characteristic polynomial $\chi_{A_mA_m^*}$ converges to the characteristic polynomial $\chi_{AA^*}$. By the continuity of the roots of a polynomial \cite{Pilaud} we see that $\sLog(A_m) \tend{m}{+\infty} \sLog(A)$. That is, $\sLog$ is continuous.
\end{proof}
\begin{proposition}
	$\sLog$ is closed.
\end{proposition}
\begin{proof}
Let $F \subset \GL_n(\C)$ be a closed subset and let $(x_m)_{m \in \N}$ be a sequence of elements of $\sLog(F)$ that converges to some $x \in \R^n$. Then for any $m \in \N$, we can find $A_m \in \GL_n(\C)$ with $x_m = \sLog(A_m)$. The sequence $(x_m)$ converges so it is bounded. By Lemma \ref{LEM:Norm sLog}, the sequences $(A_m)$ and $\left(A_m^{-1}\right)$ are bounded. Thus,  after going to a subsequence, we can assume, $A_m \tend{m}{+\infty} A$ and $A_m^{-1} \tend{m}{+\infty} B$ for some matrices $A$ and $B$ in $\Mat_n(\C)$. By the continuity of the product, $AB = I_n$ so $A$ is invertible with $A^{-1} = B$. As $F$ is closed, $A \in F$. We see that $\d x = \lim_{m \rightarrow +\infty} x_m = \lim_{m \rightarrow +\infty} \sLog(A_m) = \sLog(A)$ by the continuity of $\sLog$ (Proposition \ref{PRO:sLog continuous}). $A \in F$ so $x \in \sLog(F)$, which proves the proposition.
\end{proof}
\begin{corollary}
	Any matrix amoeba is closed.
\end{corollary}
\begin{remark}
Note that $sLog$ is in fact a proper map. This is because being proper is equivalent to being closed and the inverse image of any singleton be compact. Lemma \ref{LEM:Norm sLog} implies that the inverse image, under $\sLog$, of any singleton is bounded and it is closed by continuity of $\sLog$.
\end{remark}

Next, we describe matrix amoebas in terms of classical amoebas (in the torus).
\begin{proposition}\label{PRO:Union amoebas}
For any ideal $I \subset \C[\GL_n]$
we have 
\[
\sA(I) = \bigcup_{(U,V) \in \U(n) \times \U(n)} \A(\Psi_{U,V}(I)).\]
\end{proposition}
\begin{proof}
\noindent\framebox{$\subset$} If $x \in \sA(I)$, $x$ can be written as $(\ln(\lambda_1),\ldots,\ln(\lambda_n))$ where the $\lambda_k$ are the singular values of a matrix $A \in \V(I)$. Therefore, we can write $A$ as $U\, \diag(\lambda)\, V^*$ with $U$ and $V$ unitary. Then for all $f \in I$ we have:
\[
0 = f(A) = f(U\, \diag(\lambda) \, V^*) = \Psi_{U,V}(f)(\lambda).
\]
Thus $\lambda \in \V(\Psi_{U,V}(I))$ which shows that $x = \Log(\lambda) \in \A(\Psi_{U,V}(I))$.\\

\noindent\framebox{$\supset$} Let $U$ and $V$ be unitary matrices and $x \in \A(\Psi_{U,V}(I))$. We can write $x$ as $\Log(z) = (\ln\abs{z_1},\ldots,\ln\abs{z_n})$ where $z \in \V(\Psi_{U,V}(I))$. For all $f \in I$, we have $0 = \Psi_{U,V}(f)(z) = f(U\,\diag(z)\, V^*)$. Let $A = U\, \diag(z)\, V^*$. The singular values of $A$ are the square roots of the eigenvalues of $AA^* = U\, \diag(z)\, \overline{\diag(z)}U^* = U\, \diag\left(\abs{z_1}^2,\ldots,\abs{z_n}^2\right)\, U^*$, i.e., $\abs{z_1},\ldots,\abs{z_n}$. As for all $f \in I$, $f(A) = 0$, we deduce that $(\abs{z_1},\ldots,\abs{z_n})$ is in $\V(f)$ so $x = \Log(z) \in \sA(I)$.
\end{proof}
Even thought the union in the above proposition is over an uncountable set, it is still useful. For example we can use it prove the following.
\begin{proposition}
	For any $f \in \C[\GL_n(\C)]$, the connected components of $\sA(f)^C$ are convex.
\end{proposition}
\begin{proof}
By the proposition \ref{PRO:Union amoebas}, we know that $\d \sA(f)^C = \bigcap_{(U,V) \in \U(n)} \A(\Psi_{U,V}(f))^C$. Let $C$ be a connected component of $\sA(f)^C$ and let $x,y \in C$. The set $\sA(f)^C$ is open so its connected components are path connected. 
For all unitary matrices $U$, $V$ we have $C \subset \sA(f)^C \subset \A(\Psi_{U,V}(f))^C$. This shows that $x$ and $y$ belong to the same path connected component in $\A(\Psi_{U,V}(f))^C$ which is convex. Therefore the line segment joining $x$ and $y$ lies in $\A(\Psi_{U,V}(f))^C$, for any unitary matrices $U$, $V$. It follows that this line segment lies in $ \sA(f)^C$ as required.
\end{proof}

\section{Matrix amoebas of hypersurfaces and Ronkin function}\label{SEC:Amoebas}
In this section we study matrix amoebas for hypersurfaces. 
In particular, we generalise the notion of Ronkin function. It was, as its name suggests it, introduced by Ronkin in \cite{Ronkin}. It is a powerful tool to study the shape of amoebas of hypersurfaces. In particular, Passare and Rullgård used it in \cite{Passare} to study the spine of the amoebas which gives an easy way to compute its global shape and its homology. We will see how to extend these to matrix amoebas.

\subsection{Definitions and results for the torus $(\C^*)^n$}
Let $f$ be a Laurent polynomial. 
One defines the \emph{Ronkin function} of $f$ by:
\[
    R_f : \fonction{\R^n}{\R}{x}{\frac{1}{(2\i\pi)^n}\int_{[0,2\pi]^n} \ln\abs{f(\e^{x_1 + \i\theta_1},\ldots,\e^{x_n + \i\theta_n})} \, \diff\theta_1\cdots\diff\theta_n.}
    \]

It can be rewritten as:
\[
R_f(x) = \int_{\T^n} \ln\abs{f(\e^{x_1}\lambda_1,\ldots,\e^{x_n}\lambda_n)} \, \diff\mu(\lambda_1,\ldots,\lambda_n)
\]
where $\mu$ is the unique probability Haar measure on the compact Lie group $\T^n = (S^1)^n$ where $S^1$ denotes the unit circle. An important property of the Ronkin function is that it is convex on $\R^n$, affine on every connected component of $\A(f)^C$, and conversely, if $x \in \A(f)^C$, $R_f$ is not affine on any open neighborhood of $x$. Moreover, consider the order function $\nu_f$ given by (see \cite{Forsberg}):
\[
\nu_f : \fonction{\Gamma(\A(f)^C)}{\R^n}{C}{\nabla R_f(x) \textrm{ for some $x \in C$}}
\]
Then $\nu_f$ is injective and we have: the set of vertices of $\Newt(f) \subset \Img(\nu_f) \subset \Newt(f) \cap \Z^n$ (recall that $\Newt(f)$ is the Newton polytope of $f$ defined as the convex hull of exponents of monomials appearing in $f$). We refer to \cite{Passare,Forsberg} for several interesting results in this regard. The vector $\nu(C)$ is called the \emph{order of the connected component} $C \subset \A(f)^C$.

\subsection{Definitions for $\GL_n(\C)$}
From now on, unless otherwise stated, $f$ is an element of $\C[\GL_n(\C)]$. As the unitary group $\U(n)$ is a compact Lie group, there is a unique probability measure $\mu$ (the Haar measure) that is invariant under left-right multiplication.

\begin{definition}
We define the \emph{Ronkin function} of $f$ by:
    \[
    R_f : \fonction{\R^n}{\R \cup \{-\infty\}}{x}{\int_{\U(n)}\int_{\U(n)} \ln\abs{f(U\exp(\diag(x))V^*)} \, \diff\mu(U)\diff\mu(V)}
    \]
Note that the set of all the $U\,\exp(\diag(x))\,V^*$ $\sLog^{-1}(x)$ which is compact by the properness of $\sLog$. So $f$ is bounded on this set. It follows that the defining integral of $R_f$ is finite or $-\infty$. 
\end{definition}

We will also need to look at the coefficients of the Laurent polynomial $\Psi_{A,B}(f)$.
\begin{definition}
    Let $Q_m \in \C\!\left[\GL_n(\C) \times \GL_n(\C) \right]$ be defined by:
    \[
    f(A\diag(z)B^{-1}) = \sum_{m \in \Z^n} Q_m(A,B)z^m,
    \]
    for any $z \in T$ and invertible matrices $A$ and $B$.
\end{definition}
The regular functions $Q_m$ will be important in the study of matrix amoebas. 
We define the support of $f$ and its matrix Newton polytope using the $Q_m$.
\begin{definition}
    The \emph{support} $S_f$ of $f \in \GL_n(\C)$ is the set of $m \in \Z^n$ such that $Q_m$ is not identically zero. The \emph{matrix Newton polytope} of $f$, $\sNewt(f)$ is the convex hull of $S_f$.
\end{definition}
The support $S_f$ (respectively the polytope $\sNewt(f)$) coincides with $S_{\Psi_{U,V}(f)}$ (respectively $\sNewt(\Psi_{U,V}(f))$), for generic choices of unitary matrices $U$ and $V$. More precisely, we have the following.
\begin{proposition}\label{PRO:Sf = SPsiUVf}
    For almost every pair $(U, V)$ of unitary matrices (with respect to the Haar measure on $U(n) \times U(n)$), we have $S_f = S_{\Psi_{U,V}(f)}$.
\end{proposition}
\begin{proof}
The claim follows from Lemma \ref{LEM:Zero negligible} applied to the $Q_m$.
\end{proof}

\subsection{Some properties of the Ronkin function}
The following expresses the matrix Ronkin function in terms of the classical Ronkin functions. It will be useful as it allows us to reduce statements about the matrix Ronkin function to those of classical Ronkin function.  
\begin{proposition}\label{PRO:Expression Rf}
    For all $f \in \C[\GL_n]$, for all $x \in \R^n$,
    \[
    R_f(x) = \int_{\U(n) \times \U(n)} R_{\Psi_{U,V}(f)}(x) \, \diff\mu^2(U,V),
    \]
    where $d\mu^2$ denotes the Haar measure on $U(n) \times U(n)$.
\end{proposition}
\begin{proof}
Take $\Omega \in \U(n)$. Substituting $U$ by $U \Omega$, by the change of variable formula, we have: 
\[
R_f(x) = \int_{\U(n) \times U(n)} \ln\abs{f(U\Omega\,\exp(\diag(x))\,V^*)} \, \diff\mu^2(U,V).
\]
Noting that any diagonal matrix whose coefficients are in the unit circle $\S^1$ is unitary, we can rewrite $R_f$ as:  
\begin{align*}
    R_f(x) & = \int_{\U(n)\times U(n)} \ln\abs{f(U\exp(\diag(x))V^*)} \, \diff\mu^2(U,V),\\
    & = \frac{1}{(2\pi)^n}\int_{\T^n}\int_{\U(n) \times \U(n)} \ln\abs{f(U\exp(\diag(\i\theta))\exp(\diag(x))V^*)} \, \diff\mu^2(U,V) \, d\theta,\\
    & = \int_{\U(n) \times \U(n)} \frac{1}{(2\pi)^n}\int_{\T^n} \ln\abs{f(U\exp(\diag(x + \i\theta))V^*)} \, d\theta \, \diff\mu^2(U,V),\\
    & = \int_{\U(n) \times \U(n)} R_{\Psi_{U,V}(f)}(x) \, \diff\mu^2(U,V).
\end{align*}
This finishes the proof.
\end{proof}
\begin{proposition}\label{PRO:Ronkin basic properties}
    \
    \begin{itemize}
        \item[(a)] $R_f$ is convex.
        \item[(b)] $R_f$ has real values or is identically equal to $-\infty$.
        \item[(c)] For all $f$, $g$, $R_{f + g} = R_f + R_g$.
        \item[(d)] $R_{\det}(x) = \mathbf{1} \cdot x$, where $\mathbf{1} = (1, \ldots, 1)$.
    \end{itemize}
\end{proposition}
\begin{proof}
The first part immediately follows from the Proposition \ref{PRO:Expression Rf} and the fact that the Ronkin function for Laurent polynomials is convex. The second part follows from the continuity of the Ronkin function (convex implies continuous). The two last parts follow from simple computation.
\end{proof}

\begin{proposition}
    $R_f$ is invariant under permutations of the coordinates.
\end{proposition}
\begin{proof}
Let $\sigma \in \mathcal{S}_n$. We denote by $P_\sigma$ the permutation matrix associated with $\sigma$. It is in particular a unitary matrix. Recall that for all $z \in \C^n$, $\diag(\sigma \cdot z) = P_\sigma\diag(z)P_\sigma^*$. By the change of variable $U \mapsto UP_\sigma$ and $V \mapsto VP_\sigma$,
\begin{align*}
    R_{f}(\sigma \cdot x) & = \int_{\U(n) \times \U(n)} \ln\abs{U\exp(\diag(\sigma \cdot x))V^*} \, \diff\mu^2(U,V)\\
    & = \int_{\U(n) \times \U(n)} \ln\abs{UP_\sigma\exp(\diag(x))P_\sigma^*V^*} \, \diff\mu^2(U,V)\\
    & = \int_{\U(n) \times \U(n)} \ln\abs{U\exp(\diag(x))V^*} \, \diff\mu^2(U,V)\\
    & = R_f(x),
\end{align*}
which proves the proposition.
\end{proof}
Therefore, $R_f$ can be seen as a function of $\R^n/\mathcal{S}_n$. The classical Ronkin function is important in the study of amoebas of Laurent polynomials because it contains the information about where the connected components of the complement of the amoeba are. Namely, the Ronkin function is affine on each connected component of complement of an amoeba. We have an analogues result for matrix amoebas.

\begin{proposition}\label{PRO:Ronkin affine}
    The Ronkin function is affine on every connected component of $\sA(f)^C$. In particular, it is not identically equal to $-\infty$.
\end{proposition}
\begin{proof}
Let $C \subset \A(f)^C$ be a connected component and let $x \in C$. It means that for all unitary matrices $U,V$, $x \notin \A(\Psi_{U,V}(f))$ so $R_{\Psi_{U,V}(f)}$ is affine (thus smooth). It is clear that $U,V \mapsto \nabla R_{\Psi_{U,V}(f)}(x)$ is continuous. Moreover, this function takes its values in $\Z^n$ so it is actually constant. Let $m(x) \in \Z^n$ its value. Still by an argument of discreteness/continuity, we deduce that $m(x) = m$ is constant when $x$ browses $C$. Therefore, $R_f$ is affine over $C$ and its gradient is $m \in \Z^n$.
\end{proof}
Proposition \ref{PRO:Ronkin basic properties}(b) and Proposition \ref{PRO:Ronkin affine} imply that if $\sA(f) \neq \R^n$, $R_f > -\infty$ over the whole space. We remark that the classical Ronkin function for Laurent polynomials has finite values for every non-zero polynomial. We conjecture that if $f \neq 0$, $R_f > -\infty$.

In the next subsection, we will see more advanced results that will give us information about the Ronkin function. 

\subsection{The order function and its image}
Proposition \ref{PRO:Ronkin affine} allows us to define the order function $\nu_f$ for $f \in \C[\GL_n]$.
\begin{definition}\label{DEF:nuf}
    We define the \emph{order function} by:
    \[
    \nu_f : \fonction{\Gamma(\sA(f)^C)}{\Z^n}{C}{\nabla R_f(x) \textrm{ for some $x \in C$}}.
    \]
    We call $
    nu_f(C)$, the \emph{order of the connected component} $C$.
\end{definition}
The proof of the proposition \ref{PRO:Ronkin affine} tells us that for all unitary matrices $U,V$ and for all $x \in C \subset \sA(f)^C$ with $C$ connected, $\nabla R_f(x) = \nabla R_{\Psi_{U,V}(f)}(x)$, which will be very useful to determine the order of the connected components of $\sA(f)^C$.
\begin{proposition}\label{PRO:Injectivity nuf}
    $\nu_f$ is injective.
\end{proposition}
\begin{proof}
Let $C_1$ and $C_2$ be connected components of $\sA(f)^C$ such that $\nu_f(C_1) = \nu_f(C_2)$. We call $m$ this quantity. It implies that for all unitary matrices $U,V$, $\nu_{\Psi_{U,V}(f)}(C_1) = \nu_{\Psi_{U,V}(f)}(C_2)$. As $R_{\Psi_{U,V}(f)}$ is convex and its gradient is $m$ over $C_1$ and $C_2$, we deduce that $\nabla R_{\Psi_{U,V}(f)} = m$ over $C_3 = \Conv(C_1 \cup C_2)$ which is open. Therefore, any point $x \in C_3$ is in $\A(\Psi_{U,V}(f))^C$ because $R_{\Psi_{U,V}(f)}$ is affine around $x$. It is true for all $U,V$ so $C_3 \subset \sA(f)^C$. We deduce that $C_1 = C_2 = C_3$. $\nu_f$ is injective.
\end{proof}
\begin{proposition}\label{PRO:Image nuf}
    The image of $\nu_f$ is included in $\d \bigcap_{(U,V) \in \U(n) \times \U(n)} \Newt(\Psi_{U,V}(f)) \cap \Z^n$.
\end{proposition}
\begin{proof}
We proved that for all unitary $U,V$, $\nu_f(C) \in \Img(\nu_{\Psi_{U,V}(f)}) \subset \Newt(\Psi_{U,V}(f)) \cap \Z^n$ so $\Img(\nu_f)$ is included in $\d \bigcap_{(U,V) \in \U(n) \times \U(n)} \Newt(\Psi_{U,V}(f)) \cap \Z^n$.
\end{proof}
Knowing which of the $Q_m$ vanish on $\U(n) \times \U(n)$ is useful to eliminate quickly some points of $S_f$ that are not in $\Img(\nu_f)$. In fact, contrary to the case of Laurent polynomials where all the vertices $v$ of the Newton polytope have an associated connected component of order $v$, at most two of the vertices of $\sNewt(f)$ can have an associated component. We want to use the lemma \ref{LEM:Homogeneity} to prove that when $m \notin \Z\mathbf{1}$ is a vertex, $m \notin \Img(\nu_f)$. For this, we need some properties on the $Q_m$.
\begin{proposition}\label{PRO:Properties of the Qm}
    For all $m \in \Z^n$, if $\lambda \in T$, for all invertible matrices $A$ and $B$,
    \[
    Q_m(A\diag(\lambda),B) = \lambda^mQ_m(A,B), \qquad Q_m(A,B\diag(\lambda)) = \lambda^{-m}Q_m(A,B)
    \]
    and for all permutation matrix $P$, $Q_m(AP,BP) = Q_{Pm}(A,B)$.
\end{proposition}
\begin{proof}
For any invertible $A,B$ and any $z \in T$ and any $\lambda \in T$,
\begin{align*}
    \sum_{m \in \Z^n} Q_m(A\diag(\lambda),B^{-1})z^m & = f(A\diag(\lambda)\diag(z)B)\\
    & = f(A\diag(\lambda_1z_1,\ldots,\lambda_nz_n)B^{-1})\\
    & = \sum_{m \in \Z^n} Q_m(A,B)\lambda^mz^m,
\end{align*}
so by uniqueness of the $Q_m$, for every $m \in \Z^n$, $Q_m(A\diag(\lambda),B) = \lambda^mQ_m(A,B)$. Same thing for $B$.

And if $\sigma \in \mathcal{S}_n$, let $P_\sigma$ be the associated permutation matrix. We have,
\begin{align*}
    \sum_{m \in \Z^n} Q_m(AP_\sigma,BP_\sigma)z^m & = f(AP_\sigma\diag(z)P_\sigma^{-1}B^{-1})\\
    & = f(A\diag(\sigma \cdot z)B^{-1})\\
    & = \sum_{m \in \Z^n} Q_m(A,B)(\sigma \cdot z)^m\\
    & = \sum_{m \in \Z^n} Q_{P_\sigma m}(A,B)(\sigma \cdot z)^{P_\sigma m}\\
    & = \sum_{m \in \Z^n} Q_{P_\sigma m}(A,B)z^m,
\end{align*}
so for every $m \in \Z^n$, $Q_m(AP_\sigma,BP_\sigma) = Q_{P_\sigma m}(A,B)$.
\end{proof}
\begin{corollary}
    $S_f$ and $\sNewt(f)$ are permutation invariant.
\end{corollary}

\subsection{Geometry of hypersurfaces matrix amoebas}
In this section we study the shape of matrix amoebas of hypersurfaces, and in particular, the connected components of their complements and their maximal cones. Thanks to Proposition \ref{PRO:Injectivity nuf} ($\nu_f$ is defined at definition \ref{DEF:nuf}), we know that any connected component of $\sA(f)^C$ is associated with exactly one point of $\sNewt(f) \cap \Z^n$. We will treat vertices of $\sNewt(f)$ separately from the other points.

Recall that when $\Delta \subset \R^n$ is a convex polytope and $F$ is a face of $\Delta$, the normal cone associated to $F$ is $C_F(\Delta)$ (or $C_F$ when there is no ambiguity) defined as $\{w \in \R^n|\forall x \in F, \forall y \in \Delta, w \cdot x \leq w \cdot y\}$, which is a cone (a rational one if $\Delta$ is rational). See \cite[Section 2.3]{MacLagan} for more details about convex geometry. Notice that if $F$ has dimension $d$, $C_F$ has dimension $n - d$.
\begin{proposition}\label{PRO:Maximal cone}
    For all connected component $E \subset \sA(f)^C$, if we set $m = \nu_f(E) \in \sNewt(f) \cap \Z^n$ and $F$ the smallest face of $\sNewt(f)$ that contains $m$ (\textit{i.e.} the only face whose $m$ belongs to the relative interior of), then $C_F$ is the recession cone of $E$, which means that,
    \begin{itemize}
        \item[(a)] $\forall x \in E, x + C_F \subset E$.
        \item[(b)] $\forall$ cone $C, C_F \subsetneq C \Rightarrow \forall x \in \R^n, x + C \not\subset E$.
    \end{itemize}
\end{proposition}
\begin{proof}
To prove the proposition, we need its Laurent polynomial counterpart that can be found in \cite[Proposition 2.6]{Forsberg} 
Let $x \in E$ and let $E_{U,V}$ be the connected component of $\A(\Psi_{U,V}(f))^C$ where $x$ belongs to. We know that for all $U,V$, $\nu_{\Psi_{U,V}(f)}(E_{U,V}) = \nu_f(E) = m$ (in particular, $m$ belongs to all the $\Newt(\Psi_{U,V}(f))$) and $E$ is the intersection of the $E_{U,V}$. By \cite[Proposition 2.6]{Forsberg}, for every unitary $U$, $V$, we have $x + C_{F_{U,V}}(\Newt(\Psi_{U,V}(f))) \subset E_{U,V}$ where $F_{U,V}$ is the smallest face of $\Newt(\Psi_{U,V}(f))$ that contains $m$. As $\Newt(\Psi_{U,V}(f)) \subset \sNewt(f)$, $C_F = C_F(\sNewt(f)) \subset C_{F_{U,V}}(\Newt(\Psi_{U,V}(f)))$, we see that $x + C_F \subset E_{U,V}$. It is true for every $U,V$ so $x + C_F \subset E$. Conversely,suppose $C$ strictly contains $C_F$. Consider unitary $U_0$ ,$ V_0$ such that $\Newt(\Psi_{U,V}(f)) = \sNewt(f)$ (we know there exists at least one by Proposition \ref{PRO:Equivalence Newton polytopes}). We have, still by \cite[Proposition 2.6]{Forsberg}, that for all $x \in \R^n$, $x + C \not\subset E_{U_0,V_0} \supset E$. It proves the proposition.
\end{proof}
\begin{corollary}
    It implies that the bounded components are exactly the ones whose order is an interior point of $\sNewt(f)$ because the only face $F$ such that $C_F$ is bounded is $F = \sNewt(f)$.
\end{corollary}
\begin{proposition}\label{PRO:Image nu, vertices}
    If $v \in \sNewt(f)$ is a vertex, there are two possibilities:
    \begin{itemize}
        \item[(a)] $v \in \Z\mathbf{1}$. In this case, $v \in \Img(\nu_f)$. Moreover, $\mathbf{1}$ or $-\mathbf{1}$ belongs in $C_{\{v\}}$. In particular, $E$ is unbounded.
        \item[(b)] $v \notin \Z\mathbf{1}$. In this case, $v \notin \Img(\nu_f)$.
    \end{itemize}
\end{proposition}
\begin{proof}
Assume $v \in \Z\mathbf{1}$. If there is another point $m \in \sNewt(f) \cap \Z^n$ such that $v \cdot \mathbf{1} = m \mathbf{1}$, then $m \notin \Z\mathbf{1}$ (or it would be equal to $v$) so the convex hull of $\mathcal{S}_n \cdot m$ is an $(n - 1)$-dimensional polytope contained in $\sNewt(f)$ (because the action of $\mathcal{S}_n$ on $\{\mathbf{1}\}^\perp$ is irreducible) that contains $v$ because $\d v = \sum_{\sigma \in \mathcal{S}_n} \frac{1}{n!}\sigma \cdot m$. It implies that $v$ is not a vertex of $\sNewt(f)$, which contradicts our hypothesis. Therefore, for all $m \in \sNewt(f)$, $m \neq v \Rightarrow m \cdot \mathbf{1} \neq m \cdot v$. With the same kind of argument, we could prove that $v \cdot \mathbf{1} - m \cdot \mathbf{1}$ is always positive or always negative when $m$ browses $\sNewt(f) \cap \Z^n$. We will assume without loss of generality that it is always positive. We set $\delta > 0$ the minimum taken by the quantity $v \cdot \mathbf{1} - m \cdot \mathbf{1}$. We will need it later in the proof.

Now, decompose $f$ as $\d \sum_{d \in \Z} f_d$ where the $f_d$ are homogeneous polynomials of degree $d$, all zero except a finite number of them. Let for all $d$, and for all $A,B,z$, $\d f_d(A\diag(z)B^{-1}) = \sum_{m \in S_{f_d}} Q_m^{(d)}(A,B)z^m$. For all $\lambda \in \C$, we have,
\[
f_d(\lambda A\diag(z)B^{-1}) = \sum_{m \in S_{f_d}} \lambda^{m \cdot \mathbf{1}}Q_m^{(d)}(A,B)z^m = f_d(A\diag(z)B^{-1}) = \lambda^d\sum_{m \in S_{f_d}} \lambda^dQ_m^{(d)}(A,B)z^m.
\]
Thus, for all $m \in S_{f_d}$, $m \cdot \mathbf{1} = d$. It implies that $\d S_f = \biguplus_{d \in \N} S_{f_d}$ and for all $m \in S_f$, the $Q_m$ of $f$ is the $Q_m$ of $f_{m \cdot \mathbf{1}}$. In particular, if we set $v = N\mathbf{1}$ with $N \in \Z$ (because $v \in \Z\mathbf{1}$), we have $v \cdot \mathbf{1} = nN$ and $f_{nN}(A\diag(z)B^{-1}) = Q_v(A,B)z^v$. Notice that for all invertible diagonalisable matrix $A = P\diag(z)P^{-1}$, we have,
\[
f_{nN}(A) = Q_v(P,P)z^v = f_{nN}(PP^{-1})z^v = f_{nN}(I_n)\det(A)^N.
\]
By density, this equality remains true for any invertible matrix. In fact $f_{nN} = \alpha\det^N$ for some $\alpha \in \C^*$. In particular, $Q_v$ does not vanish on $\U(n) \times \U(n)$. Let $(U,V) \in \U(n) \times \U(n)$ and $x \in \R_+^*$.
\begin{align*}
    \abs{\Psi_{U,V}(f)(\e^x\mathbf{1}) - Q_v(U,V)(\e^x\mathbf{1})^v} & = \abs{\sum_{m \neq v} Q_m(U,V)\e^{xm \cdot \mathbf{1}}}\\
    & \leq \sum_{m \neq v} \max_{\U(n) \times \U(n)}\{\abs{Q_m}\}\e^{x(v \cdot \mathbf{1} - \delta)} \textrm{ by definition of $\delta$.}\\
    & = \leq \e^{-x\delta}\sum_{m \neq v} \max_{\U(n) \times \U(n)}\{\abs{Q_m}\}\e^{xv \cdot \mathbf{1}}\\
    & < \abs{\alpha}\e^{xv \cdot \mathbf{1}} \textrm{ if $x$ is large enough.}\\
    & = \abs{Q_v(U,V)(\e^x\mathbf{1})^v}.
\end{align*}
By \cite[Proposition 2.7]{Forsberg} applied to $\Psi_{U,V}$, it implies that $x\mathbf{1} \in \A(\Psi_{U,V}(f))^C$ and $\nabla R_{\Psi_{U,V}(f)}(x\mathbf{1}) = v$ for every $U,V$ so $x\mathbf{1} \in \sA(f)^C$ and $\nabla R_f(x\mathbf{1}) = v$. Let $E$ be the connected component where $x$ belongs. We have by definition of $\nu_f$ that $\nu_f(E) = v$ so $v \in \Img(\nu_f)$. It is true for any $x$ large enough so $E$ is unbounded and contains a half-line included in $\R\mathbf{1}$. By the proposition \ref{PRO:Maximal cone}, it means that $\mathbf{1}$ or $-\mathbf{1}$ belongs in $C_{\{v\}}$.

If $v \notin \Z\mathbf{1}$, by the proposition \ref{PRO:Properties of the Qm} and the lemma \ref{LEM:Homogeneity} applied to $A \mapsto \det(A)^NQ_v(A,I_n)$ for $N$ large enough, $Q_v$ vanishes on $\U(n) \times \U(n)$. As $v$ is a vertex of $\sNewt(f)$, if we consider some $U_0,V_0$ unitary such that $Q_v(U_0,V_0) = 0$, then $v \notin \Newt(\Psi_{U_0,V_0}(f))$ so there is no connected component of order $v$ in the complement of the amoeba of $\Psi_{U_0,V_0}(f)$. It implies that there is no connected component of order $v$ in the complement of the spherical amoeba of $f$, which proves the proposition.
\end{proof}
\begin{remark}
    This is a difference between classical amoebas and matrix amoebas. For matrix amoebas, every vertex of the Newton polytope is in the image of the order map. In the case of matrices, there are at most two because there are at most two vertices on the line $\Z\mathbf{1}$.
\end{remark}
\begin{remark}
    For a Laurent polynomial $f$, it is possible but rare to find a connected component of $\A(f)^C$ of order $m$ if the coefficient of $f$ at the monomial $x^m$ is 0. Nisse call those coefficient "virtually non zero" and it has consequences as \cite[Theorem 1.2]{Nisse}. Therefore, as each $Q_m$ vanishes on $\U(n) \times \U(n)$ when $m \notin \Z\mathbf{1}$, it seems possible that for all matrix polynomial $f$, $\Img(\nu_f) \subset \Z\mathbf{1}$. We have neither been able to prove it nor to find a counterexample.
\end{remark}

\subsection{Some examples}

We will see three examples to illustrate this section. The $\alpha_k$ in these examples denote non zero complex constants. They all are in dimension $n = 2$ except the first example, which is in dimension $n \geq 2$. Notice that the support of $f$ can easily be computed thus we won't detail the computation of the $S_f$.
\begin{example}
    $f \in \Mat_n(\C)^* \subset \C[\GL_n]$ is a linear map.
\end{example}
We will see that in this case, the amoeba of $f$ is the whole space $\R^n$. Indeed, we know that $f(A)$ can be written as $\tr(MA)$ for some matrix $M$, by Riesz's representation theorem. Let us use $M$'s singular values decomposition $M = U_0DV_0^*$. Let $P$ be the matrix of a permutation which does not have a fixed point (that exists because $n \geq 2$). We have for all $z \in T$,
\[
\Psi_{V_0P,U_0}(f)(z) = \tr(U_0DV_0^*V_0P\diag(z)U_0^*) = \tr(P\diag(z)D) = 0.
\]
We deduce that $\Psi_{V_0P,U_0}(f) = 0$ so its amoeba is $\R^n$. Therefore, $\sA(f) = \R^n$. Notice that it is different than for Laurent polynomial where the amoeba is the whole space if and only if the polynomial is null.

\begin{example}\label{EX:Ex 2}
    $n = 2$ and $f : A \mapsto \alpha_1a_{11}^2 + \alpha_2\det(A)$.
\end{example}
Here, $S_f = \sNewt(f) \cap \Z^2 = \{(2,0),(1,1),(0,2)\}$. Notice that it is contained in $\{(1,1)\}^\perp + (1,1)$. It is actually easy to verify that in general $\sNewt(f)$ is contained in an affine hyperplane of $\R^n$ parallel to $\{1\}^\perp$ if and only if $f$ is homogeneous. Here, $f$ is indeed homogeneous of degree 2. $(2,0)$ and $(0,2)$ are vertices of $\sNewt(f)$ that are not in $\Z(1,1)$ so by the proposition \ref{PRO:Image nu, vertices}, they do not have associated components. It implies that $\sA(f)^C$ is empty or connected. By the proposition \ref{PRO:Maximal cone}, if $\sA(f)^C \neq \O$, the biggest cone contained in it is $C_{\sNewt(f)} = \R(1,1)$ and it is open so it is of the form $\{(a + b,-a + b)|\abs{b} < r\}$ for some $r > 0$. Same thing with $r = 0$ if $\sA(f)^C = \O$. The fact that this set is empty or not depends on the $\alpha_k$.
\begin{proposition}
    $\sA(f)^C = \O$ if and only if $\abs{\alpha_2} > \abs{\alpha_1}$.
\end{proposition}
\begin{proof}
In this case, $\sA(f)^C \neq \O$ if and only if it contains 0 if and only if $f$ vanishes on $\U(2)$ if and only if $f$ vanishes on $\textup{SU}(2)$ (by homogeneity). Let $\d U = \begin{pmatrix} a & b \\ -\overline{b} & \overline{a} \end{pmatrix} \in \textup{SU}(2)$ with $\abs{a}^2 + \abs{b}^2 = 1$. $f(U) = \alpha_1a^2 + \alpha_2 = 0$ if and only if $a$ is a square root of $\d -\frac{\alpha_2}{\alpha_1}$. It is possible to find such an $a$ if and only if $\abs{\alpha_2} \leq \abs{\alpha_1}$, which proves the proposition.
\end{proof}
Notice that it provides a second example of non-zero polynomial whose amoeba is the whole space (when $\abs{\alpha_2} \leq \abs{\alpha_1}$). The figure \ref{FIG:Example 2} shows the Newton polytope of $f$ and its amoeba in the case where $\abs{\alpha_2} > \abs{\alpha_1}$.
\begin{figure}[H]
    \centering
    \includegraphics[width = 0.3\textwidth]{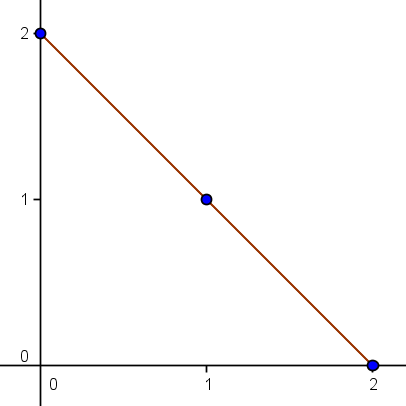}
    \includegraphics[width = 0.3\textwidth]{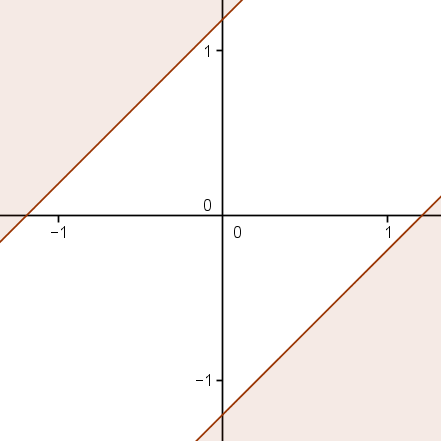} 
    \caption{The Newton polytope (left) and the amoeba (right) of a polynomial of the example \ref{EX:Ex 2}.}
    \label{FIG:Example 2}
\end{figure}

\begin{example}\label{EX:Ex 3}
    $n = 2$ and $f : A \mapsto \alpha_1 + \alpha_2a_{11} + \alpha_3\det(A) + \alpha_4a_{22}\det(A)$.
\end{example}
In this case, $S_f = \sNewt(f) \cap \Z^2 = \{(0,0),(1,0),(0,1),(1,1),(2,0),(0,2)\}$. $(1,0)$, $(0,1)$, $(2,0)$ and $(0,2)$ all are vertices that do not belong to $\Z(1,1)$ so they do not have associated component by the proposition \ref{PRO:Image nu, vertices}. $(0,0)$ is a vertex on $\Z(1,1)$ so it has an associated unbounded component $C_0$ whose maximal cone is the quarter of the plane $\{(x,y) \in \R^2|\max\{x,y\} \leq 0\}$ by propositions \ref{PRO:Image nu, vertices} and \ref{PRO:Maximal cone}. $(1,1)$ can or can not have an associated component in function of the $\alpha_k$. If it has one, it is bounded because $(1,1)$ is an interior point of $\sNewt(f)$. let $C_1$ be this component, or $C_1 = \O$ if $(1,1) \notin \Img(\nu_f)$.
\begin{proposition}
    If the polynomial $\abs{\alpha_4}t^3 - \abs{\alpha_3}t^2 + \abs{\alpha_2}t + \abs{\alpha_1}$ takes negative values on $\R_+^*$ (which means that $\abs{\alpha_3}$ is large enough compared to the others $\abs{\alpha_k}$), $C_1 \neq \O$. Conversely, if $\abs{\alpha_3}^2 < 4\abs{\alpha_2}\abs{\alpha_4}$, then $C_1 = \O$.
\end{proposition}
\begin{proof}
Assume the first condition holds. Let $r \in \R_+^*$ such that $\abs{\alpha_4}r^3 - \abs{\alpha_3}r^2 + \abs{\alpha_2}r + \abs{\alpha_1} < 0$ and let $x = \ln(r)$. Let $A$ be a matrix whose singular values are both $\e^x = r$. Then $A = rU$ where $U \in \U(2)$ so,
\[
\abs{\alpha_3\det(A)} - \abs{\alpha_1 + \alpha_2a_{11} + \alpha_4a_{22}\det(A)} \geq \abs{\alpha_3}r^2 - \abs{\alpha_1} - \abs{\alpha_2}r - \abs{\alpha_4}r^3 > 0.
\]
It implies that $(x,x) \notin \sA(f)$ and by \cite[Proposition 2.7]{Forsberg} applied to each $\Psi_{U,V}(f)$, the order of the component containing $(x,x)$ is $(1,1)$. $C_1 \neq \O$.

Conversely, assume now that $C_1 \neq \O$. Then, every $\A(\Psi_{U,V}(f))^C$ has a connected component of order $(1,1)$. Is is in particular the case of $g = \Psi_{I_2,I_2}(f)$. For every $(z_1,z_2) \in (\C^*)^2$,
\[
g(z_1,z_2) = \alpha_1 + \alpha_2z_1 + \alpha_3z_1z_2 + \alpha_4z_1z_2^2.
\]
Therefore, $S_g = \Newt(g) \cap \Z^2 = \{(0,0),(1,0),(1,1),(1,2)\}$. As $\A(g)^C$ has a component of order $(1,1)$ and every other lattice point of $\Newt(g)$ are vertices, each lattice point of $g$ has an associated connected component. The spine of $\A(g)$ is the tropical curve of a tropical polynomial of the form $\max\{c_{(0,0)},c_{(1,0)} + x,c_{(1,1)} + x + y,c_{(1,2)} + x + 2y\}$ and by \cite[Theorem 2]{Passare} and \cite[Theorem 3]{Passare}, $c_{(1,1)}$ is given by,
\[
c_{(1,1)} = \ln\abs{\alpha_3} + \Re\left(\sum_{k \in K} \frac{(-k_3 - 1)!}{k_2!k_4!}(-1)^{k_3 - 1}\alpha_2^{k_2}\alpha_3^{k_3}\alpha_4^{k_4}\right),
\]
where $K = \{(k_2,k_3,k_4) \in \Z^3|k_2 \geq 0, k_3 < 0, k_4 \geq 0, k_2 + k_3 + k_4 = 0, k_3 + 2k_4 = 0\}$. We can parameterize the elements of $K$ by $K = \{(q,-2q,q)|q \in \N^*\}$ so,
\[
c_{(1,1)} = \ln\abs{\alpha_3} - \Re\left(\sum_{q \in \N^*} \frac{(2q - 1)!}{q!^2}\left(\frac{\alpha_2\alpha_4}{\alpha_3^2}\right)^q\right),
\]
and $\d \frac{(2q - 1)!}{q!^2} = \frac{\binom{2q}{q}}{2q} \eq{q}{+\infty} \frac{4^q}{2\sqrt{\pi}q^{3/2}}$ (this is a consequence of the Stirling formula) so this series has a convergence radius of $\d \frac{1}{4}$. It implies that $\abs{\alpha_3}^2 \geq 4\abs{\alpha_2}\abs{\alpha_4}$, which proves the proposition.
\end{proof}
The figure \ref{FIG:Example 3} shows the Newton polytope of $f$ and the shape of its amoeba in the case where $C_1 \neq \O$. Notice that we aren't able to compute spherical amoebas yet so this picture does not represent an approximation of the real amoeba of $f$ for certain values of the $\alpha_k$ but is only a representation of what its amoeba looks like (in particular, the picture and the real amoeba have the same homotopy).
\begin{figure}[H]
    \centering
    \includegraphics[width = 0.3\textwidth]{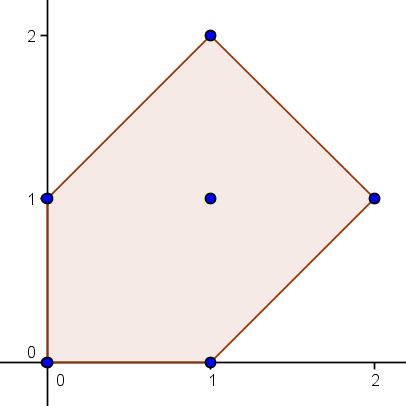}
    \includegraphics[width = 0.3\textwidth]{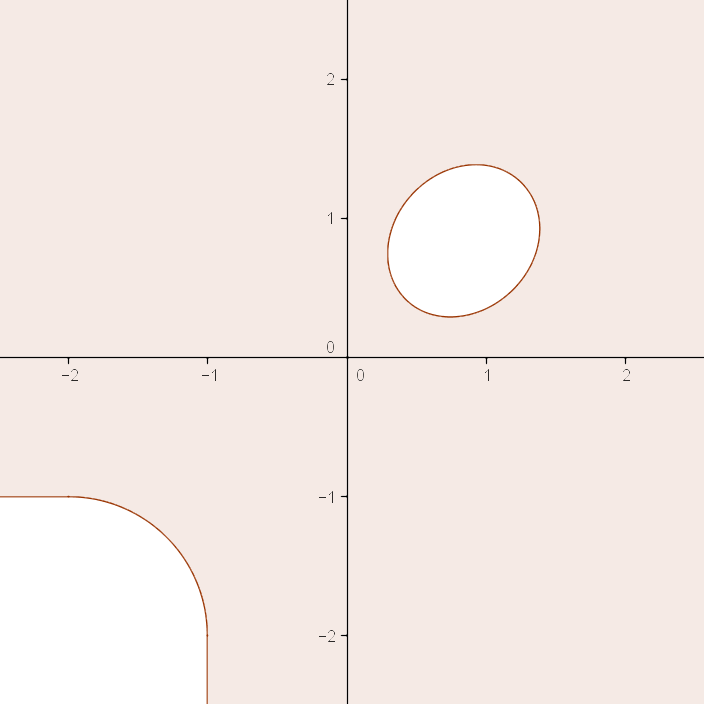} 
    \caption{The Newton polytope (left) and the shape of the amoeba (right) of a polynomial of the example \ref{EX:Ex 3}.}
    \label{FIG:Example 3}
\end{figure}

\section{Definition of matrix Newton polytope using representation theory}\label{SEC:Representations}

The purpose of this section is to show that the Newton polytope $\sNewt(f)$ of a polynomial $f \in \C[\GL_n]$ and an other definition of Kapranov \cite{Kapranov} of the Newton polytope of $f$ (that we will call $\Delta(f)$) coincide. Let $G$ be a reductive algebraic group (we will take $G = \GL_n(\C)$), then $H = \{(g,g)|g \in G\}$ is the stabilizer of $G^2$ and $G \cong G^2/H$ is a homogeneous space. Moreover, the Borel subgroup $B \subset G^2$ acts on it with an open orbit, thus $G$ is a spherical variety ($B$ is the set of couples of invertible matrices $(U,L)$ with $U$ upper triangular and $L$ lower triangular when $G = \GL_n(\C)$).

By the Peter-Weyl theorem, we have that,
\[
\C[G] = \bigoplus_{\lambda \in \Lambda} W_\lambda,
\]
where $\Lambda$ is the Weyl chamber associated to $G$ ($\Lambda = \{x \in \R^n|x_1 \leq \cdots \leq x_n\}$ when $G = \GL_n(\C)$) and $W_\lambda$ is the irreducible representation of $G^2$ with highest $B$-weight $\lambda$ by the highest weight theorem. Moreover, for each $\lambda$, there exists, up to a rescalling, a unique $B$-weight vector of weight $\lambda$ (that belongs to $W_\lambda$) and $W_\lambda \cong V_\lambda \otimes V_\lambda^*$ as representations of $G^2$ where $V_\lambda$ is the irreducible representation of $G$ of highest weight $\lambda$. More details about representation theory in \cite{Fulton}.

Let $W$ be the Weyl group of $G$ ($W = \mathcal{S}_n$ when $G = \GL_n(\C)$). The moment polytope $\Delta(f)$ of $f$ is the convex hull of the $W$-orbits of the $\lambda \in \Lambda$ such that the projection of $f$ on $W_\lambda$ parallel to $\d \bigoplus_{\mu \neq \lambda} W_\mu$ is non-zero in $\R^n$. It is in particular stable by $W$. Kapranov calls it the Newton polytope in \cite{Kapranov}.

This section is a digression about representation theory of $\GL_n(\C)$ that has for purpose to show that for any $f \in \C[\GL_n]$, $\sNewt(f) = \Delta(f)$. This section is not linked to amoebas, neither tropical geometry.

\subsection{Some convex geometry}

First of all, we need results about convex geometry.

\begin{definition}
    For every $\lambda \in \Lambda = \{\lambda \in \R^n|\lambda_1 \leq \cdots \leq \lambda_n\}$, we define $C(\lambda)$ as the convex hull of the $\mathcal{S}_n$-orbit of $\lambda$ in $\R^n$. When $(x,y) \in \Lambda^2$, we say that $x \preccurlyeq y$ if $C(x) \subset C(y)$.
\end{definition}
We want to show that $\preccurlyeq$ is a (partial) order relation over $\Lambda$. It is obviously reflexive and transitive. Let us show that it is anti-symmetric.

\begin{lemma}\label{LEM:Characterisation of C(x)}
    For every $x \in \Lambda$, $C(x)$ is the set of $y \in \R^n$ such that for all $K \subset \{1, \ldots, n\}$, $y \cdot \ind_K \leq x \cdot \varpi_{\abs{K}}$ with equality when $K = \{1, \ldots, n\}$, where $\ind_K \in \R^n$ is such that for all $i$, $(\ind_K)_i = \ind_K(i)$ and for all $1 \leq k \leq n$, $\varpi_k = \ind_{[\![n + 1 - k,n]\!]}$.
\end{lemma}
\begin{proof}\ \\
\framebox{$\subset$} If $y \in C(x)$, then it can be written as $\d \sum_{\sigma \in \mathcal{S}_n} \alpha_\sigma\sigma \cdot x$ where the $\alpha_\sigma$ are non-negative real numbers whose sum equals 1. Therefore, for all $K \subset \{1, \ldots, n\}$,
\[
y \cdot \ind_K = \sum_{\sigma \in \mathcal{S}_n} \alpha_\sigma(\sigma \cdot x) \cdot \ind_K \leq \sum_{\sigma \in \mathcal{S}_n} \alpha_\sigma x \cdot \varpi_{\abs{K}} = x \cdot \varpi_{\abs{K}}.
\]
The inequality $(\sigma \cdot x) \cdot \ind_K \leq x \cdot \varpi_{\abs{K}}$ comes from the fact that the coefficients of $x$ increase. Moreover, it becomes an equality when $K = \{1, \ldots, n\}$, which proves that $y \cdot \mathbf{1} = x \cdot \mathbf{1}$.\\

\noindent\framebox{$\supset$} We will prove it by induction on $n$. It is trivial for $n = 1$. Let $n \geq 2$ and $y \in \R^n$ that verifies the hypothesis. Up to permuting its coordinates, we can assume that $y \in \Lambda$. By the hypothesis, $y_n = y \cdot \ind_{\{n\}} \leq x \cdot \varpi_1 = x_n$ and $\d x_1 \leq \frac{1}{n}x \cdot \mathbf{1} = \frac{1}{n}y \cdot \mathbf{1} \leq y_n$. We have $x_1 \leq y_n \leq x_n$ so there exists an integer $1 \leq q \leq n$ such that $x_q \leq y_n \leq x_{q + 1}$. Let $\tau$ be the transposition between $q$ and $q + 1$. Let $x' = (1 - t)x + t\tau \cdot x$ for some $0 \leq t \leq 1$. As the support of $\tau$ is $\{q,q + 1\}$, for all $k \notin \{q,q + 1\}$, $x'_k = x_k$. Moreover,
\begin{align*}
    t = 0 & \Rightarrow x'_q = x_q \leq y_n,\\
    t = 1 & \Rightarrow x'_q = x_{q + 1} \geq y_n.
\end{align*}
Therefore, for the right value of $t$, $x'_q = y_n$. The sum of all the coefficients is stable by permutation so we have that $x'_{q + 1} = x_q + x_{q + 1} - y_n$. Now, let
\[
\tilde{x} = \begin{pmatrix} x_1 \\ \vdots \\ x_{q - 1} \\ x_q + x_{q + 1} - y_n \\ x_{q + 2} \\ \vdots \\ x_n \end{pmatrix}, \quad \tilde{y} = \begin{pmatrix} y_1 \\ \vdots \\ y_{n - 1} \end{pmatrix}.
\]
They are both vectors of $\R^{n - 1}$. Let us show they verify the hypothesis of the lemma. They both belong in $\Lambda$ (trivial) so it is enough to check that for all $k$, $\tilde{y} \cdot \varpi_k \leq \tilde{x} \cdot \varpi_k$ with equality when $k = n - 1$. The equality when $k = n - 1$ is trivial. Now, let $1 \leq k \leq n - 1$.
\begin{align*}
    k \leq n - q - 1 \Rightarrow (\tilde{y} - \tilde{x}) \cdot \varpi_k & = y_{n - k} + \cdots + y_{n - 1} - x_{n + 1 - k} - \cdots - x_n\\
    & = y \cdot \varpi_k + y_{n + 1 - k} - y_n - x \cdot \varpi_k\\
    & \leq 0,\\
    k \geq n - q \Rightarrow (\tilde{y} - \tilde{x}) \cdot \varpi_k & = y_{n - k} + \cdots + y_{n - 1} - x_{n - k} - \cdots - x_{q - 1} - (x_q + x_{q + 1} - y_n) - x_{q + 2} - \cdots - x_n\\
    & = y \cdot \varpi_{k + 1} - x \cdot \varpi_{k + 1}\\
    & \leq 0.
\end{align*}
By induction, it proves that $\tilde{y} \in C(\tilde{x}) \Leftrightarrow \tilde{y} \preccurlyeq \tilde{x}$. Notice that $x' = (\tilde{x},x_n) \preccurlyeq x$ by construction and $y = (\tilde{y},y_n) \preccurlyeq x'$ because $\tilde{y} \preccurlyeq \tilde{x}$. Finally, by transitivity, $y \preccurlyeq x$, which proves the lemma.
\end{proof}

\begin{proposition}\label{PRO:Order}
    $\preccurlyeq$ is anti-symmetric, thus an order relation over $\Lambda$.
\end{proposition}
\begin{proof}
By the lemma \ref{LEM:Characterisation of C(x)}, if $x \preccurlyeq y$ and $y \preccurlyeq x$ and if $x$ and $y$ both belong to $\Lambda$, then for all $k$, $x \cdot \varpi_k = y \cdot \varpi_k$, which implies that $x = y$.
\end{proof}
\begin{proposition}\label{PRO:Linearity of C}
    For every $(x,y) \in \Lambda^2$ and for every $\alpha \in \R$, $C(\alpha x) = \alpha C(x)$ and $C(x + y) = C(x) + C(y)$.
\end{proposition}
\begin{proof}
The first part is trivial, let us focus on the second one. First of all, the Minkowski sum of two convex is convex thus $C(x) + C(y)$ is convex. As it contains by definition every $\sigma \cdot (x + y)$, we have that $C(x + y) \subset C(x) + C(y)$. Let us show the reciprocal. As both sets are convex, it is enough to show that $C(x + y)$ contains every $\sigma \cdot x + \rho \cdot y$ where $\sigma$ and $\rho$ are permutations. Indeed, for all set $K \subset \{1, \ldots, n\}$,
\[
(\sigma \cdot x + \rho \cdot y) \cdot \ind_K = (\sigma \cdot x) \cdot \ind_K + (\rho \cdot y) \cdot \ind_K \leq (x + y) \cdot \varpi_k,
\]
and $(\sigma \cdot x + \rho \cdot y) \cdot \mathbf{1} = (x + y) \cdot \mathbf{1}$ so by the characterisation given by the lemma \ref{LEM:Characterisation of C(x)}, $\sigma \cdot x + \rho \cdot y \in C(x + y)$.
\end{proof}

\subsection{The highest weight vector $v_\lambda$}

The purpose of the subsection is to study the irreducible $\GL_n(\C)^2$-representations $W_\lambda$. We won't be able to compute them explicitly but we can at least compute the only (up to a non-zero scalar) $B$-weight vector $v_\lambda \in \C[\GL_n]$ of weight $\lambda \in \Lambda$. We will focus particularly on the Newton polytopes $\sNewt$ of the polynomials of $W_\lambda\backslash\{0\}$.

\begin{definition}
    When $I$ and $J$ are subsets of $\{1, \ldots, n\}$ of same cardinality, we define $\det_{I,J}(A)$ as the determinant of the sub-$A$-matrix of $A$ where we only kept the lines indexed in $I$ and the columns indexed in $J$. In particular, $\det_{I,J}$ is a polynomial over $\Mat_n(\C)$. We also define $\det_k$ as $\det_{[\![k + 1 - n,n]\!],[\![k + 1 - n,n]\!]}$.
\end{definition}
\begin{definition}
    We define for every $\lambda \in \Lambda$,
    \[
    v_\lambda = \prod_{k = 1}^n \det_k^{\lambda_{n - k + 1} - \lambda_{n - k}},
    \]
    where $\lambda_0 = 0$ by convention. As $\lambda \in \Lambda$, for every $k \geq 1$, $\lambda_k \leq \lambda_{k + 1}$ so $v_\lambda \in \C[\GL_n]$. Notice that in particular, $v_{\varpi_k} = \det_k$ for all $k$.
\end{definition}
\begin{proposition}
    $v_\lambda$ is the $B$-weight vector of weight $\lambda$ in $\C[\GL_n]$ in the sens that for every $(U,L) \in B$ the Borel subgroup of $\GL_n(\C)^2$, $(U,L) \cdot v_\lambda = \lambda(U)\lambda^*(L)v_\lambda$ where
    \[
    \lambda : U \mapsto \prod_{k = 1}^n u_{kk}^{\lambda_k}, \quad \lambda^* : L \mapsto \prod_{k = 1}^n l_{kk}^{-\lambda_k}.
    \]
    are the associated weight and its dual.
\end{proposition}
\begin{proof}
Let $(U,L) \in B$ and $X \in \GL_n(\C)$. Let us decompose them into blocs,
\[
U = \begin{pmatrix} * & * \\ 0 & U' \end{pmatrix}, \quad L = \begin{pmatrix} * & 0 \\ * & L' \end{pmatrix}, \quad X = \begin{pmatrix} * & * \\ * & X' \end{pmatrix},
\]
where $L'$, $U'$ and $X'$ and $k \times k$ matrices. We compute that
\[
UXL^{-1} = \begin{pmatrix} * & * \\ * & U'X'(L')^{-1} \end{pmatrix}.
\]
Therefore, $\det_k(UXL^{-1}) = \det(U'X'(L')^{-1}) = \varpi_k(U)\varpi_k^*(L)\det_k(X)$ where $\varpi_k \in \Lambda$ is the weight $(0,\ldots,0,1,\ldots,1)$ with $k$ ones. By product, $v_\lambda(UXL^{-1}) = \lambda(U)\lambda^*(L)v_\lambda(X)$, which proves the proposition.
\end{proof}

Now, let us determine the spherical Newton polytope of the $v_\lambda$.

\begin{proposition}\label{PRO:sNewt(v_lambda)}
    $\sNewt(v_\lambda) = C(\lambda)$ when $\lambda = \varpi_k$ ($1 \leq k \leq n$) or $-\varpi_n$.
\end{proposition}
\begin{proof}
We have that $v_{\varpi_k} = \det_k$. This proposition is a direct consequence of the Binet-Cauchy formula,
\[
\det_{I,J}(AB) = \sum_{\abs{K} = k} \det_{I,K}(A)\det_{K,J}(B)
\]
when $A$ and $B$ are matrices and $\abs{I} = \abs{J} = k$. In particular, with $A \leftarrow A\diag(z)$ and $B \leftarrow B^{-1}$ when $A$ and $B$ are invertible and $z$ is in the torus,
\[
\det_k(A\diag(z)B^{-1}) = \sum_{\abs{K} = k} \det_{[\![k + 1 - n,n]\!],K}(A)\det_{K,[\![k + 1 - n,n]\!]}(B)z^{\ind_K}.
\]
Therefore, $S_{\det_k} = \{\ind_K|\abs{K} = n\} = \mathcal{S}_n \cdot \varpi_k$ so by definition, $\sNewt(v_{\varpi_k}) = C(\varpi_k)$. For $\lambda = -\varpi_n$, $v_{\lambda} = \det^{-1}$ so $\sNewt(v_{-\varpi_n}) = \{-\mathbf{1}\} = C(-\varpi_n)$.
\end{proof}

\begin{lemma}\label{LEM:sNewt(fg)}
    If $f,g$ are two matrix polynomials, $\sNewt(fg) = \sNewt(f) + \sNewt(g)$ (in the sense of Minkowski).
\end{lemma}
\begin{proof}
It is well-known that this formula is true for classical Laurent polynomial and their Newton polytope. For any $h \in \C[\GL_n]$, the set of $(A,B) \in \GL_n(\C)^2$ such that $\sNewt(h) = \Newt(\Psi_{A,B}(h))$ is Zariski open, thus dense. It implies that for generic invertible matrices $A,B$, $\sNewt(f) = \Newt(\Psi_{A,B}(f))$, $\sNewt(g) = \Newt(\Psi_{A,B}(g)$ and $\sNewt(fg) = \Newt(\Psi_{A,B}(fg))$ so,
\[
\sNewt(fg) = \Newt(\Psi_{A,B}(fg)) = \Newt(\Psi_{A,B}(f)\Psi_{A,B}(g)) = \Newt(\Psi_{A,B}(f)) + \Newt(\Psi_{A,B}(g)) = \sNewt(f) + \sNewt(g).
\]
\end{proof}

\begin{proposition}
    For every $\lambda$, $\sNewt(v_\lambda) = C(\lambda)$.
\end{proposition}
\begin{proof}
Any $\lambda \in \Lambda$ can be written as $\d \lambda = \sum_{k = 1}^n \alpha_k\varpi_k$ where the $\alpha_k$ are non-negative integers except $\alpha_n$ which is a relative integer. Notice that $\d v_\lambda = \prod_{k = 1}^n v_{\varpi_k}^{\alpha_k}$ so by lemma \ref{LEM:sNewt(fg)}, proposition \ref{PRO:sNewt(v_lambda)} and proposition \ref{PRO:Linearity of C},
\[
\sNewt(v_\lambda) = \sum_{k = 1}^n \alpha_k\sNewt(v_{\varpi_k}) = \sum_{k = 1}^n \alpha_kC(\varpi_k) = C(\lambda).
\]
This remains true even when $\alpha_n < 0$.
\end{proof}

\subsection{Proof of the equivalence}

We now have enough tools to prove the wanted proposition \ref{PRO:Equivalence Newton polytopes}.

\begin{proposition}\label{PRO:sNewt W_lambda}
    Every non-zero polynomial $f \in W_\lambda$ verifies $S_f = S_{v_\lambda}$, thus $\sNewt(f) = \sNewt(v_\lambda) = C(\lambda)$.
\end{proposition}
\begin{proof}
First of all, it is clear that the support of a matrix polynomial is stable by the action of $\GL_n(\C)^2$ by definition. As $W_\lambda = \Span(\GL_n(\C) \cdot v_\lambda)$, any polynomial $f \in W_\lambda$ verifies $S_f \subset S_{v_\lambda}$ by sum. Now, consider some $m \in S_{v_\lambda}$ and $V = \{f \in W_\lambda|m \notin S_f\}$. It is clearly a vector space because $m \notin S_f \Leftrightarrow$ the $Q_m$ of $f$ is null. Moreover, $V$ is stable by the action of $\GL_n(\C)^2$, it is a representation of this group. But $W_\lambda$ is irreducible so $V = \{0\}$ or $W_\lambda$. However, as $m \in S_{v_\lambda}$, $v_\lambda \notin V$ thus $V = \{0\}$. It is true for any $m$ so the proposition is proven.
\end{proof}

\begin{proposition}\label{PRO:Equivalence Newton polytopes}
    For all polynomial $f \in \C[\GL_n]$, $\sNewt(f) = \Delta(f)$.
\end{proposition}
\begin{proof}
Let $f \in \C[\GL_n]$. Write $\d f = \sum_{\lambda \in A} f_\lambda$ where $A \subset \Lambda$ is a finite set and for all $\lambda$ in $A$, $f_\lambda \in W_\lambda\backslash\{0\}$. Let us show that $\sNewt(f) = \Delta(f)$. As $\sNewt(f) = \Conv(S_f)$ and $\Delta(f) = \Conv(\mathcal{S}_n \cdot A)$ and they are both stable by permutation, it is enough to show that $S_f \subset \Conv(A)$ and $A \subset \Conv(S_f)$.\\

\noindent\framebox{$\subset$} If $\lambda \in S_f$. It means that $Q_\lambda^{(f)} \neq 0$ so $Q_\lambda^{(f_\mu)} \neq 0$ for at least one $\mu \in A$. It means that $\lambda \in \sNewt(f_\mu) = C(\mu)$ by the proposition \ref{PRO:sNewt W_lambda}. $\lambda \in C(\mu) \subset \Conv(A) = \Delta(f)$.\\

\noindent\framebox{$\supset$} If $\lambda \in A$, let $\lambda_M \in A$ such that $\lambda \preccurlyeq \lambda_M$ and $\lambda_M$ is a maximal element of $A$. Such a $\lambda_M$ exists because $A$ is finite. Let us show that $\lambda_M \in S_f$. Indeed, $Q_{\lambda_M}^{(f_{\lambda_M})} \neq 0$ and for all $\mu \in A\backslash\{\lambda_M\}$, $\lambda_M \not\preccurlyeq \mu$ because $\lambda_M$ is a maximal element of $A$, thus $\lambda_M \notin C(\mu) = \sNewt(f_\mu)$ so $Q_{\lambda_M}^{(f_\mu)} = 0$. It implies that $Q_{\lambda_M}^{(f)} = Q_{\lambda_M}^{(f_{\lambda_M})} \neq 0$. $\lambda_M \in S_f$. It proves that $\lambda \in C(\lambda_M) \subset \Conv(S_f) = \sNewt(f)$.
\end{proof}

\section{Tropical geometry}\label{SEC:Tropical}

The goal of the is section is to generalise the theorem due to Bergman \cite{Bergman} that makes the link between the tropical variety of an ideal and its amoeba. We will also recall the definition of the tropical variety of a matrix spherical variety $Y \subset \GL_n(\C)$ which is a particular case of the definition given by Tevelev and Vogiannou \cite{Tevelev}.

\subsection{Definitions and theorem in $(K^*)^n$}

Let $K$ be an algebraically closed field of characteristic 0 endowed with a non trivial valuation that is trivial on $\Q$. Let $I$ be a proper ideal of the Laurent polynomial ring $K[X^\pm] = K\!\left[X_1^\pm,\ldots,X_n^\pm\right]$. We define the algebraic variety $Y = Y(K)$ associated to $I$ as the set of non zero vectors $z \in (K^*)^n$ where all the polynomials of $I$ vanish. For every polynomial in $K[X^\pm]$, we define its tropical version as
\[
\trop\left(\sum_{m \in \Z^n} a_mX^m\right) = x \mapsto \min\{\val(a_m) + m \cdot x|m \in \Z^n\},
\]
which is a affine by part convex function from $\R^n$ to $\R$. Notice that the $\min$ convention has been used but a similar version with a $\max$ also exists. Introduction to Tropical Geometry, by MacLagan and Sturmfel \cite{MacLagan} is a good reference for tropical geometry. We define the tropical hypersurface of any tropical polynomial $P$ as the set of points where the minimum is reached at least twice \textit{i.e.} the set where it is not differentiable. We call it $\V(P)$. When $Y = \V(I) \subset (K^*)^n$ is a very affine variety, we have the fundamental theorem of tropical algebra,
\[
\bigcap_{f \in I} \V(\trop(f)) = \overline{\{\val(z_1),\ldots,\val(z_n)|z \in Y\}}.
\]
We call this set $\trop(Y)$. Notice that it only depends on $Y$ and not on the choice of $I$. There is a third definition using initial ideals (more \cite[Theorem 3.2.3]{MacLagan}). Assume now that $K = \overline{\mathcal{K}}$ the set of complex Puiseux series, endowed with a valuation,
\[
\val : \sum_{m \in S} a_mt^m \mapsto \min(S) \textrm{ where $S$ is non empty and the $a_m$ are non zero compex numbers,}
\]
and $\val(0) = +\infty$. This field is algebraically closed of characteristic 0 and the valuation is non trivial, but is trivial over $\Q$. Therefore, the previous theorem holds. We $I \subset \C[X^\pm]$, let $Y = Y(\C)$ be the variety associated to $I$ and $Y(\overline{\mathcal{K}}) \subset (\overline{\mathcal{K}}^*)^n$ be the variety associated with the ideal of $\overline{\mathcal{K}}[X^\pm]$ generated by the elements of $I$. We have the Bergman's theorem,
\[
\rho\A(Y) \tend{\rho}{0^+} -\trop(Y) \textrm{ in the sens of Kuratowski,}
\]
and $\trop(Y)$ is a finite union of polyhedral cones of codimension at least 1. Moreover, when $I = (f)$ is principal, $-\trop(Y)$ is the normal cone of the Newton polytope of $f$. See \cite[Section 2.3]{MacLagan} for an introduction to convex geometry and the definition of the normal cone of a convex polytope. Let us extend it to matrices.

\subsection{Definitions in $\GL_n(\C)$}

We will work exclusively on $\C$ and $\overline{\mathcal{K}}$ since the notions of singular values and invariant factors are hardly generalisable to any field.
\begin{definition}
    Let $\RR = \{z(t) \in \overline{\mathcal{K}}|\val(z(t)) \geq 0\}$.
\end{definition}
$\RR$ is a ring and its invertible elements are $\RR^\times = \{z(t) \in \overline{\mathcal{K}}|\val(z(t)) = 0\} = \{z(t) \in \RR|z(0) \neq 0\}$. Notice that $\frak{m} = \RR\backslash\RR^\times$ is an ideal, so $\RR$ is local with maximal ideal $\frak{m}$ and the residue field of $\overline{\mathcal{K}}$ is $\RR/\frak{m} = \C$. But the most important is that $\RR$ is integral with $\mathrm{Frac}(\RR) = \overline{\mathcal{K}}$. The natural way to extend to notion of coordinated-wise valuation of a vector of complex Puiseux series to matrices is to use the invariant factors given by the Smith normal form (we will see that it coincides with the definition given in \cite{Tevelev}). However, Smith's theorem requires to work on a principal ideal domain and $\RR$ is not principal (it is not even Noetherian as $\frak{m}$ is not finitely generated). We need to extend Smith's theorem.
\begin{proposition}[Smith normal form]
    For all Puiseux series matrix $A(t) \in \Mat_n(\overline{\mathcal{K}})$, there exists matrices $(P(t),Q(t)) \in \GL_n(\RR)$ and a diagonal matrix $D(t)$ such that $A(t) = P(t)D(t)Q^{-1}(t)$. Moreover, the valuations of the diagonal elements of $D(t)$ are unique up to permutation. We call them invariant factors.
\end{proposition}
\begin{proof}
Let $A(t)$ be a Pusieux series matrix and $G$ be the group generated by $\{q \in \Q|t^q \textrm{ appears in } A(t)\}$. Exponents in Puiseux series all have a common denominator and $A(t)$ has a finite number of coefficients so $G$ is discrete. Let $\d K = \left\{\sum_{g \in G} a_gt^g \in \overline{\mathcal{K}}\right\}$. $K$ is a field and the valuation inherited from the valuation on $\overline{\mathcal{K}}$ is discrete on $K$ because $\Gamma_{\val_{|K}} = G$. As $K$ is a field with a discrete valuation, the ring $R = \{z(t) \in K|\val(z(t)) \geq 0\} \subset \RR$ is a principal ideal domain. As all the coefficients of $A(t)$ belong to $K$, we can apply the Smith normal form theorem. There exists $(P(t),Q(t)) \in \GL_n(R) \subset \GL_n(\RR)$ and $D(t)$ a diagonal matrix such that $A(t) = P(t)D(t)Q^{-1}(t)$. In particular, the coefficients of $P(t)$ and $Q(t)$ are series whose exponents all are in $G \cap \R_+$.

If $A(t) = P(t)D(t)Q^{-1}(t) = \tilde{P}(t)\tilde{D}(t)\tilde{Q}^{-1}(t)$ with $\tilde{P}(t)$ and $\tilde{Q}(t)$ invertible in $\RR$ and $\tilde{D}(t)$ diagonal, we can use the uniqueness in the Smith normal form theorem in the field $\d \tilde{K} = \left\{\sum_{g \in \tilde{G}} a_gt^g \in \overline{\mathcal{K}}\right\}$ where $\tilde{G}$ is the group generated by the exponents of the coefficients of $P(t)$, $Q(t)$, $D(t)$, $\tilde{P}(t)$, $\tilde{Q}(t)$ and $\tilde{D}(t)$ to deduce that the valuation of the diagonal coefficients in $D(t)$ are the same than in $\tilde{D}(t)$ up to permutation. It proves the proposition.
\end{proof}
It allows us to define the matrix spherical valuation of a Puiseux series matrix $A(t)$,
\begin{definition}
    \[
    \sval : \fonction{\GL_n(\overline{\mathcal{K}})}{\Q^n/\mathcal{S}_n}{A(t) = P(t)\diag(z(t))Q^{-1}(t)}{(\val(z_1(t)),\ldots,\val(z_n(t)))}.
    \]
\end{definition}
Now, let $I$ be an ideal of the ring $\C[\GL_n]$. We define the spherical variety $Y(\C)$ associated to $I$ as the set of invertible matrices $A$ where all the polynomials of $I$ vanish and the spherical variety in Puiseux series $Y(\overline{\mathcal{K}})$ as the variety associated with the ideal generated by the elements of $I$ in $\overline{\mathcal{K}}[\GL_n]$. When there is no ambiguity, we can use the notation $Y$ to talk about $Y(\C)$ as well as $Y(\overline{\mathcal{K}})$. There does not seem to be a good generalisation of the tropicalization of a matrix polynomial so we shall define the spherical tropical variety of $Y$ thanks to the spherical valuation,
\begin{definition}
    \[
    \strop(Y) = \overline{\{\sval(A(t))|A(t) \in Y(\overline{\mathcal{K}})\}}
    \]
\end{definition}
where the set $\strop_\Q(Y) = \{\sval(A(t))|A(t) \in Y(\overline{\mathcal{K}})\} \subset \Q^n/\mathcal{S}_n$ is mistaken by abuse with the set of $x \in \Q^n$ such that the orbit of $x$ under $\mathcal{S}_n$ is in $\strop_\Q(Y)$. We can make the link with classical tropical variety, and thus use if necessary the fundamental theorem of tropical algebra. Notice that this definition actually coincides with the more general definition of tropical varieties of a spherical variety by Tevlev and Vogiannou \cite[Theorem 1.3]{Tevelev}.
\begin{proposition}\label{PRO:Union trop}
    Let $I$ be a proper ideal of $\C[\GL_n]$. For any matrices $P(t),Q(t)$ that are invertible in $\GL_n(\RR)$, and $f \in \C[\GL_n]$, we extend the definition of $\Psi$,
    \[
    \Psi_{P(t),Q(t)}(f) : \fonction{(\overline{\mathcal{K}}^*)^n}{\overline{\mathcal{K}}}{z(t)}{f(P(t)\diag(z(t))Q^{-1}(t))}.
    \]
    $\Psi_{P(t),Q(t)} : \overline{\mathcal{K}}[\GL_n] \rightarrow \overline{\mathcal{K}}[X^\pm]$ is a ring morphism. Let $Y_{P(t),Q(t)} = \V(\Psi_{P(t),Q(t)}(I))$. We have
    \[
    \strop(Y) = \overline{\bigcup_{(P(t),Q(t)) \in \GL_n(\RR)} \trop(Y_{P(t),Q(t)})}
    \]
\end{proposition}
\begin{proof}
Let $x \in \Q^n$.
\begin{align*}
    x \in \strop_\Q(Y) & \Leftrightarrow x = \sval(A(t)) \textrm{ for some } A(t) \in Y.\\
    & \Leftrightarrow f(P(t)\diag(z(t))Q^{-1}(t)) = 0 \textrm{ for some } (P(t),Q(t)) \in \GL_n(\RR)^2 \textrm{ and for all $k$, } \val(z_k(t)) = x_k.\\
    & \Leftrightarrow x \in \trop_\Q(Y_{P(t),Q(t)}) \textrm{ for some } (P(t),Q(t)) \in \GL_n(\RR)^2\\
    & \Leftrightarrow x \in \bigcup_{(P(t),Q(t)) \in \GL_n(\RR)} \trop_\Q(Y_{P(t),Q(t)}),
\end{align*}
which proves the proposition, by taking the closure.
\end{proof}
As for classical amoebas, we conjecture that for every spherical variety $Y$,
\begin{conjecture}
    \[
    \rho\sA(Y) \tend{\rho}{0^+} -\strop(Y) \textrm{ in the sens of Kuratowski.}
    \]
\end{conjecture}

\subsection{A first inlcusion}

Recall the definition of Kuratowski limit : if $(E_m)_{m \in \N}$ is a family of subsets of a topological space $E$ (we can replace the discrete $m \in \N$ by a continuous variable that converge, or diverges toward $+\infty$ or $-\infty$),
\[
\liminf_{m \rightarrow +\infty} E_m = \{x \in E|\forall U \textrm{ neighborhood of $x$ }, \exists m_0 \in \N, \forall m \geq m_0, U \cap E_m \neq \O\},
\]
\[
\limsup_{m \rightarrow +\infty} E_m = \{x \in E|\forall U \textrm{ neighborhood of $x$}, \forall m_0 \in \N, \exists m \geq m_0, U \cap E_m \neq \O\}
\]
In particular, $\d \liminf_{m \rightarrow +\infty} E_m \subset \limsup_{m \rightarrow +\infty} E_m$ and when they are equal, we call $\d \lim_{m \rightarrow +\infty} E_m$ the common limit. Therefore, given a variety $Y$, the conjecture is equivalent to
\[
\limsup_{\rho \rightarrow 0^+} \rho\sA(Y) \subset -\strop(Y) \subset \liminf_{\rho \rightarrow 0^+} \rho\sA(Y).
\]
First of all, let us show the second inclusion, which is the easiest. Consider for any rational number $q$, the truncation under $q$,
\begin{definition}
    \[
    T_q : \fonction{\overline{\mathcal{K}}}{\overline{\mathcal{K}}}{\sum_{m \in S} a_mt^m}{\sum_{\underset{m \leq q}{m \in S}}{a_mt^m}}
    \]
\end{definition}
and extend it to $\Mat_n(\C)$ by truncating each coefficient. It is in both cases a $\C$-linear map. It is clear that $\RR$ is stable under $T_q$ and for all rational number $q \geq 0$, $(t^q)$ is an ideal such that for all $a(t) \in \RR$, $T_qa(t) \equiv a(t)\ [t^q]$. This remains true if we replace $\RR$ by $\Mat_n(\RR)$.
\begin{lemma}\label{LEM:Domination invariant factors}
    For any matrices $A(t) \in \GL_n(\overline{\mathcal{K}})$ and $B(t) \in \Mat_n(\overline{\mathcal{K}})$, if all the invariant factors of $A(t)$ are less or equal than all the invariant factors of $B(t)$, then $B(t)A^{-1}(t)$ has its coefficients in $\RR$.
    
    Moreover, if we replace "less or equal" by "less", $(A(t) + B(t))A^{-1}(t) \in \GL_n(\RR)$.
\end{lemma}
\begin{proof}
Let $q$ be a rational number such that all the invariant factors of $A(t)$ are less or equal than $q$ and all the invariant factors of $B(t)$ are greater or equal than $q$. Therefore, $t^{-q}B(t) \in \Mat_n(\RR)$ and the invariant factors of $A^{-1}(t)$ are opposite to the invariant factors of $A(t)$ thus $t^qA^{-1}(t) \in \Mat_n(\RR)$ so $B(t)A^{-1}(t) = t^{-q}B(t)t^qA^{-1}(t)$ which has coefficients in $\RR$.

Now, if we replace "less or equal" by "less", by using the first part of the lemma with $A(t)$ and $t^{-\varepsilon}B(t)$ for a small enough $\varepsilon > 0$, the matrix $B(t)A^{-1}(t)$ is in the ideal $(t^\varepsilon)$ of the ring $\Mat_n(\RR)$. Let $\d S(t) = \sum_{m \in \N} (-1)^m(B(t)A^{-1}(t))^m$. As $B(t)A^{-1}(t) \in (t^\varepsilon)$, this series converges in $\Mat_n(\RR)$ and it is clear that $S(t)(I_n + B(t)A^{-1}(t)) = I_n$. It implies that $I_n + B(t)A^{-1}(t) = (A(t) + B(t))A^{-1}(t)$ is invertible in $\RR$.
\end{proof}
\begin{lemma}
    For any $A(t) \in \GL_n(\overline{\mathcal{K}})$, for any $q \in \Q$ greater than every invariant factor of $A(t)$, $T_qA(t)$ is invertible and $\sval(T_qA(t)) = \sval(A(t))$.
\end{lemma}
\begin{proof}
If $A(t) \in \Mat_n(\overline{\mathcal{K}})$, $v_{\min}$ and $v_{\max}$ be its smallest and the biggest invariant factors. We know that $t^{-v_{\min}}A(t) \in \Mat_n(\RR)$ so we can write it as
\[
t^{-v_{\min}}A(t) = P(t)\diag(t^{v_1 - v_{\min}},\ldots,t^{v_n - v_{\min}})Q^{-1}(t),
\]
where $P(t)$ and $Q(t)$ are both invertible in $\RR$ and $\sval(A(t)) = (v_1,\ldots,v_n)$. Moreover, if $q = v_{\max} - v_{\min} + \varepsilon$ for some positive rational $\varepsilon$, $T_q\diag(t^{v_1 - v_{\min}},\ldots,t^{v_n - v_{\min}}) = \diag(t^{v_1 - v_{\min}},\ldots,t^{v_n - v_{\min}})$ so
\[
t^{-v_{\min}}T_{q + v_{\min}}A(t) = T_q(t^{-v_{\min}}A(t)) \equiv T_qP(t)\diag(t^{v_1 - v_{\min}},\ldots,t^{v_n - v_{\min}})T_qQ^{-1}(t)\ [t^q]
\]
and $q + v_{\min} = v_{\max} + \varepsilon$ so by multiplying the previous equality by $t^{v_{\min}}$,
\[
T_{v_{\max} + \varepsilon}A(t) = T_qP(t)\diag(t^{v_1},\ldots,t^{v_n})T_qQ^{-1}(t) + B(t)
\]
where $B(t) \in (t^{v_{\max} + \varepsilon})$. Let $\tilde{A}(t) = T_qP(t)\diag(v_1,\ldots,v_n)T_qQ^{-1}(t)$. $q$ is positive by definition so we have that $\det(T_q(P(t))(0)) = \det(P(0)) \in \C^*$ because $P$ is invertible in $\RR$. It implies that $T_qP(t)$ is also invertible in $\RR$. Same thing with $Q^{-1}(t)$.

Therefore, $\tilde{A}(t)$ is invertible and its invariant factors are the $v_k$ which are all dominated by $v_{\max}$, itself strictly dominated by all the invariant factors of $B(t)$. By the lemma \ref{LEM:Domination invariant factors}, $T_{v_{\max} + \varepsilon}A(t) = \tilde{A}(t) + B(t)$ and $\tilde{A}(t)$ have the same image in $\RR$ up to isomorphism so they have the same invariant factors, which proves the lemma.
\end{proof}
Using the same reasoning with the lemma \ref{LEM:Domination invariant factors}, we deduce that,
\begin{corollary}\label{COR:Same invariant factors}
    For any $(A(t),B(t)) \in \GL_n(\overline{\mathcal{K}})^2$, for any $q \in \Q$ greater than every invariant factor of $A(t)$, if $A(t) \equiv B(t)\ [t^q]$, $B(t)$ is invertible and $\sval(B(t)) = \sval(A(t))$.
\end{corollary}
\begin{lemma}\label{LEM:Field extension}
    Let $F/K$ be a field extension with $K$ and $F$ both algebraically closed. Let $I_F \subset F[X_1^\pm,\ldots,X_n^\pm]$ and $I_K = K[X_1^\pm,\ldots,X_n^\pm]I_F$ which is an ideal of $K[X_1^\pm,\ldots,X_n^\pm]$. If $\val : K \rightarrow \R$ is a valuation, such as $\val(K) = \val(F)$ (we call $\Gamma$ this dense subgroup of $\R$), $\trop(\V(I_F)) = \trop(\V(I_K))$.
\end{lemma}
\begin{proof}\ \\
\framebox{$\subset$} If $x \in \trop_\Gamma(\V(I_F))$, for all $1 \leq k \leq n$, $x_k = \val(z_k)$ for some $z \in \V(I_F) \subset (F^*)^n$. $I_K$ is generated by $I_F$ so $z \in \V(I_K)$ thus $x \in \trop_\Gamma(\V(I_K))$.\\

\noindent\framebox{$\supset$} If $x \in \trop_\Gamma(\V(I_K))$, by the fundamental theorem of tropical geometry, $\d x \in \Gamma \cap \bigcap_{f \in I_K} \V(\trop(I_K))$. As $I_F \subset I_K$, $\d x \in \Gamma \cap \bigcap_{f \in I_K} \V(\trop(I_K)) = \trop_\Gamma(\V(I_F))$.\\

We proved that $\trop_\Gamma(\V(I_F)) = \trop_\Gamma(\V(I_K))$ thus $\trop(\V(I_F)) = \trop(\V(I_K))$ by taking the closure.
\end{proof}
Now, let us prove the last lemma we need for the first inclusion,
\begin{lemma}\label{LEM:Convergent matrix}
    For any ideal $I \subset \C[\GL_n]$, for any $A(t) \in \V(I)$, there exists a $B(t) \in \GL_n(F) \cap \V(I)$ that verifies $\sval(A(t)) = \sval(B(t))$ where $F$ is the set of Puiseux series which converge in $]0,\varepsilon[$ for some $\varepsilon > 0$.
\end{lemma}
\begin{proof}
Let $I$ be such an ideal and $A(t) \in \V(I)$. Let $q$ be a rational number greater than any invariant factor of $A(t)$. By Newton-Puiseux theorem, $F$ is an algebraically closed field. Let $S = \{(i,j) \in \{1, \ldots, n\}^2|(A(t) - T_qA(t))_{ij} \neq 0\}$ and for all polynomial $f \in \C[\GL_n]$,
\[
g_f : \fonction{(F^*)^S}{F}{(x_{ij}(t))_{(i,j) \in S}}{f(T_qA(t) + t^qX(t)) \textrm{ where $X_{ij}(t) = x_{ij}(t)$ if $(i,j) \in S$, 0 else.}} \in F\!\left[(X_{ij})_{(i,j) \in S}^\pm\right]
\]
Let $J_F = \{g_f|f \in I\}$ and $J_{\overline{\mathcal{K}}} = J_F\overline{\mathcal{K}}\!\left[(X_{ij})_{(i,j) \in S}^\pm\right]$. We verify easily that $J_F$ is an ideal of $F\!\left[(X_{ij})_{(i,j) \in S}^\pm\right]$. By lemma \ref{LEM:Field extension}, $\trop(J_F) = \trop(J_{\overline{\mathcal{K}}}) \subset \R^S$. Moreover, by definition of the $g_f$ and $J_{\overline{\mathcal{K}}}$, if we set for all $(i,j) \in S$, $z_{ij}(t) = t^{-q}(A(t) - T_qA(t))_{ij} \neq 0$ and $x_{ij} = \val(z_{ij}(t)) > 0$ (they are positive by definition of $T_q$), we have $(z_{ij}(t)) \in \V(J_{\overline{\mathcal{K}}})$ thus $(x_{ij}) \in \trop(\V(J_{\overline{\mathcal{K}}})) = \trop(J_F)$. Therefore, there exists Puiseux series $\left(\tilde{z}_{ij}(t)\right)_{(i,j) \in S}$ that converge in a neighborhood of $0^+$ such that for all $(i,j) \in S$, $\val\left(\tilde{z}_{ij}(t)\right) = x_{ij} > 0$ and which belong to $\V(J_F)$. Let $B(t) = T_qA(t) + t^qZ(t)$ where $Z_{ij}(t) = \tilde{z}_{ij}(t)$ if $(i,j) \in S$, 0 else. In particular, $Z \in \Mat_n(\RR \cap F)$. By construction, $B(t) \in \Mat_n(F)$, $B(t) \in \V(I)$ and $B(t) \equiv A(t)\ [t^q]$ thus $B(t) \in \GL_n(F)$ and $\sval(B(t)) = \sval(A(t))$ by lemma \ref{COR:Same invariant factors}. It proves the lemma.
\end{proof}
And finally, the wanted inclusion,
\begin{proposition}
    $\d -\strop_\Q(Y) \subset \liminf_{\rho \rightarrow 0^+} \rho\sA(Y)$.
\end{proposition}
\begin{proof}
Let $x \in -\strop_\Q(Y)$. It means that there exists $A(t) \in \GL_n(\overline{\mathcal{K}})$ such that $x = -\sval(A(t))$ and for all $f \in I$, $f(A(t)) = 0$. By lemma \ref{LEM:Convergent matrix}, we can assume without loss of generality that $A(t)$ converges on $]0,\varepsilon[$ for some $\varepsilon > 0$. We have \cite[Theorem 1.1]{Kaveh}, which has been proven in $\mathcal{K}$ but works the same way in $\overline{\mathcal{K}}$,
\[
\frac{\sLog(A(s))}{\ln(s)} \tend{s}{0^+} \sval(A(t)),
\]
so with $\d \rho = -\frac{1}{\ln(s)}$,
\[
\rho\sLog(A(\e^{-1/\rho})) \tend{\rho}{0^+} x.
\]
and all the $\sLog(A(\e^{-1/\rho}))$ (for $\rho > 0$ small enough) belong to $\sA(Y)$, which proves the inclusion.
\end{proof}
\begin{corollary}\label{COR:Inclusion1}
    $\d -\strop(Y) \subset \liminf_{\rho \rightarrow 0^+} \rho\sA(Y)$ because inferior and superior limits in the sens of Kuratowski are always closed.
\end{corollary}

\subsection{The second inclusion when $I$ is principal}

In this section, we consider $I = (f)$ a principal ideal. It implies that the $\Psi_{P(t),Q(t)}(I) = (\Psi_{P(t),Q(t)}(f))$ are also principal. Propositions \ref{PRO:Image nu, vertices} and \ref{PRO:Maximal cone} give an idea of the shape of the amoeba of $f$ in function of its Newton polytope $\sNewt(f)$. It will be helpful in order to determine its spherical tropical variety.
\begin{definition}
    If $\sNewt(f) \cap \Z\mathbf{1} = \O$, we define $C^- = C^+ = \O$. Else, consider $N^-$ (resp. $N^+$) the smallest (resp. biggest) integer such that $N^-\mathbf{1}$ (resp. $N^+\mathbf{1}$) belongs to $\sNewt(f)$. We define $C^-$ as $\inter{C_{\{N^-\mathbf{1}\}}}$ (resp. $C^+$ as $\inter{C_{\{N^+\mathbf{1}\}}}$). Notice that if $N^-\mathbf{1}$ (resp. $N^+\mathbf{1}$) is not a vertex of $\sNewt(f)$, $C^- = \O$ (resp. $C^+ = \O$).
\end{definition}
\begin{proposition}\label{PRO:Inclusion2}
    $\d \limsup_{\rho \rightarrow 0^+} \rho\sA(f) \subset \R^n\backslash(C^- \cup C^+)$.
\end{proposition}
\begin{proof}
We need to prove that every point in $C^-$ or in $C^+$ does not belong to the limit. By symmetry, it is enough to prove it for $C^+$. It is trivial is $C^+$ is empty. Assume now that $\sNewt(f) \cap \Z\mathbf{1} \neq \O$ and $N^+\mathbf{1}$ is a vertex point of $\sNewt(f)$. Let $x \in C^+ = \inter{C_{\{N^+\mathbf{1}\}}}$. By propositions \ref{PRO:Image nu, vertices} and \ref{PRO:Maximal cone}, there exists a point $y \in E$ where $E \in \Gamma(\sA(f)^C)$ and $\nu_f(E) = N^+\mathbf{1}$, and $y + C_{\{N^+\mathbf{1}\}} \subset E$.

Consider $r > 0$ such that $B(x,r) \subset C_{\{N^+\mathbf{1}\}}$. For all $\d x' \in B\left(x,\frac{r}{2}\right)$ and for all $\d 0 < \rho < \frac{r}{2\norme{y}}$ (any $\rho > 0$ if $y = 0$), $\d x' - \rho y \in B\left(x',\frac{r}{2}\right) \subset B(x,r) \subset C_{\{N^+\mathbf{1}\}}$. As this set is a cone, it is stable by product by a positive real number so $\d \frac{x'}{\rho} - y \subset C_{\{N^+\mathbf{1}\}}$, which means that $x' \in \rho(y + C_{\{N^+\mathbf{1}\}}) \subset \rho\sA(f)^C$ for all $\rho$ small enough. This is true for all $x'$ in a neighborhood of $x$ so $\d x \notin \limsup_{\rho \rightarrow 0^+} \rho\sA(f)$. It proves the proposition.
\end{proof}
We now need to prove that $\R^n\backslash(C^- \cup C^+) \subset -\strop(\V(f))$ to get the wanted inclusion. The proof is a bit handmade and needs some lemmas.
\begin{lemma}
    Let $p \geq 2$ and $(f_k)_{1 \leq k \leq p}$ a family of continuous function of a segment $[a,b]$ such that there exists $1 \leq i \neq j \leq p$ verifying for all $1 \leq k \leq p$, $f_i(a) \leq f_k(a)$ and $f_j(b) \leq f_k(b)$. Then, there exists $a \leq c \leq b$ and $1 \leq i' \neq j' \leq p$ such that for all $1 \leq k \leq p$, $f_{i'}(c) = f_{j'}(c) \leq f_k(c)$.
\end{lemma}
\begin{proof}
Just apply the intermediate values theorem to the function $\d s \mapsto f_i(s) - \min_{k \neq i}\{f_k(s)\}$.
\end{proof}
\begin{lemma}\label{LEM:Tropical shift}
    Let $p \geq 2$, $(f_k)_{1 \leq k \leq p}$ be a family of continuous functions on a segment $[a,b]$ and $(w_k)_{1 \leq k \leq p}$ a family of distinct vectors of $\R^n$ with integer coefficients. We define for all $a \leq s \leq b$ be the tropical polynomial $\d T_s : x \mapsto \min_{1 \leq k \leq p}\{f_k(s) + w_k \cdot x\}$. Let for all $k$ and $s$, $M_k(s,x) = f_k(s) + w_k \cdot x$ be the monomials. If $x \in \R^n$ is such that $T_a(x) = M_i(a,x)$ and $T_b(x) = M_j(b,x)$ for some $1 \leq i \neq j \leq p$, then $\d x \in \overline{\bigcup_{s \in [a,b] \cap \Q} \V(T_s)}$.
\end{lemma}
\begin{proof}
According to the previous lemma used with the functions $(M_k(\cdot,x))$, there exists a $a \leq c \leq b$ such that $x \in \V(T_c)$. Therefore, we can introduce $K = \{1 \leq k \leq p|T_c(x) = M_k(c,x)\}$ that contains at least two elements. Let $\d \varepsilon_0 = \min_{k \notin K}\{M_k(c,x)\} - T_c(x) > 0$. Now, let $\varepsilon > 0$ that we assume to be less than $\varepsilon_0$ without loss of generality. As the $M_k(\cdot,x)$ are continuous, there exists a $\delta > 0$ such that for all $s$, $\d \abs{s - c} \leq \delta \Rightarrow \abs{M_k(s,x) - M_k(c,x)} \leq \frac{\varepsilon}{2}$. Let $q$ be a rational number such that $\abs{q - c} \leq \delta$. Let $i_1 \neq i_2$ be such that for all $k \neq i_1$, $M_{i_1}(q,x) \leq M_{i_2}(q,x) \leq M_k(q,x)$. If $i_1 \notin K$, let $k \in K$,
\[
M_{i_1}(q,x) \geq M_{i_1}(c,x) - \frac{\varepsilon}{2} > T_c(x) - \frac{\varepsilon}{2} + \varepsilon_0 = M_k(c,x) - \frac{\varepsilon}{2} + \varepsilon_0 \geq M_k(q,x) - \varepsilon + \varepsilon_0 \geq M_k(q,x),
\]
which is a contradiction. Therefore, $i_1 \in K$. If $i_2 \notin K$, we do the same reasoning but we choose $k \in K\backslash\{i_1\}$ that exists because $\abs{K} \geq 2$. It implies by the way that $M_{i_1}(c,x) = M_{i_2}(c,x)$ so by triangular inequality and the definition of $\delta$, $M_{i_1}(q,x) - M_{i_2}(q,x) \leq \varepsilon$. Recall that the $w_k$ are distinct. Let $y = x + \alpha(w_{i_2} - w_{i_1})$ where $\d \alpha = \frac{M_{i_1}(q,x) - M_{i_2}(q,x)}{\abs{w_{i_1} - w_{i_2}}}$.
\[
M_{i_1}(q,y) - M_{i_2}(q,y) = M_{i_1}(q,x) - M_{i_2}(q,x) + (w_{i_1} - w_{i_2}) \cdot (y - x) = 0.
\]
In particular, the minimum of the tropical polynomial $T_q$ is reached for an index that is not $i_1$ (it can be $i_2$ or an other index, it does not matter). Therefore, if we use again the previous lemma with the $s \mapsto M_k(q,x + s(w_{i_2} - w_{i_1}))$, there exists a $0 \leq \beta \leq \alpha$ such that $z = x + \beta(w_{i_2} - w_{i_1}) \in \V(T_q)$. Moreover,
\[
\abs{z - x} \leq \alpha\abs{w_{i_2} - w_{i_1}} = \abs{M_{i_1}(q,x) - M_{i_2}(q,x)} \leq \varepsilon.
\]
$q \in \Q$ and it is true for every $\varepsilon$ in a neighborhood of $0^+$ so $\d x \in \overline{\bigcup_{s \in [a,b] \cap \Q} \V(T_s)}$.
\end{proof}
Now, let us prove the desired inclusion.
\begin{proposition} \label{PRO:Inclusion3, case 1}
    If $S_f \not\subset \Z\mathbf{1}$, $\d \lim_{\rho \rightarrow 0^+} \rho\, \sA(f) = -\strop(\V(f)) = \R^n\backslash(C^- \cup C^+)$ and this set is $\{\mathbf{1}\}^\perp$ if $f$ is not invertible, $\O$ else.
\end{proposition}
\begin{proof}
Assume that for $S_f \subset \Z\mathbf{1}$. By an argument similar as the proof of the one in the proof of the proposition \ref{PRO:Image nu, vertices}, we have for every $A,B$, $Q_{N\mathbf{1}}(A,B) = \alpha_N\det(AB^{-1})$ so $f = P \circ \det$ where $P$ is a one variable Laurent polynomial. We compute easily that $\sA(f) = \{\mathbf{1}\}^\perp + \A(P)$ and $\A(P)$ is the finite set of the norms of the non zero roots of $P$. $f$ invertible is equivalent to $\abs{S_f} = 1$, which is equivalent to $P$ being a monomial. In that case $\A(P) = \O$ so $\d \lim_{\rho \rightarrow 0^+} \rho\sA(f) = -\strop(\V(f)) = \O$. Else, $\A(P) \neq \O$ so $\d \lim_{\rho \rightarrow 0^+} \rho\sA(f) = \{\mathbf{1}\}^\perp$. Let $r$ be a non zero root of $P$. Notice that for all $x \in \Q^n \cap \{\mathbf{1}\}^\perp$, $f(\diag(rt^{-x_1},t^{-x_2},\ldots,t^{-x_n})) = P(r) = 0$ and $-\sval(\diag(rt^{-x_1},t^{-x_2},\ldots,t^{-x_n})) = x$. We deduce that $\{\mathbf{1}\}^\perp \subset -\strop(\V(f))$ and the reverse inclusion is given by the corollary \ref{COR:Inclusion1}.

When $f$ is invertible, $\abs{S_f} = 1$ so $C^- = C^+ = \R^n$. If $f$ is non invertible, $\abs{S_f} \geq 2$ so $C^- = \{x \in \R^n|x \cdot \mathbf{1} < 0\}$ and $C^+ = \{x \in \R^n|x \cdot \mathbf{1} > 0\}$. In both cases, $-\strop(\V(f)) = \R^n\backslash(C^- \cup C^+)$.
\end{proof}
\begin{proposition}\label{PRO:Inclusion3, case 2}
    If $S_f \not\subset \Z\mathbf{1}$, $\R^n\backslash(C^- \cup C^+) \subset -\strop(\V(f))$.
\end{proposition}
\begin{proof}
Let $x \in -\R^n\backslash(C^- \cup C^+)$. Let $m_0$ in $S_f$ that minimises its scalar product with $x$ in the sense that for all $m \in S_f$, $m_0 \cdot x \leq m \cdot x$. If there exists an $m \in S_f\backslash\{m_0\}$ such that $m_0 \cdot x = m \cdot x$, we just have to consider unitary matrices $U,V$ such that $Q_{m_0}(U,V) \neq 0$ and $Q_m(U,V) \neq 0$ and we verify easily that $x \in \V(\trop(\Psi_{U,V}(f)))$, thus we will assume that for all $m \in S_f\backslash\{m_0\}$, $m_0 \cdot x < m \cdot x$. As $-x$ is not in $C^- \cup C^+$ and $m_0$ maximises strictly its scalar product with $-x$ in $S_f$, we deduce that $m_0$ is not in $\{N^-\mathbf{1},N^+\mathbf{1}\}$. And it is clear that $m_0$ is not an interior point of the Newton polytope so $m_0 \notin \Z\mathbf{1}$ ($S_f \not\subset \Z\mathbf{1}$ so any point of $\sNewt(f) \cap \Z\mathbf{1}\backslash\{N^-\mathbf{1},N^+\mathbf{1}\}$ is an interior point of $\sNewt(f)$). It implies that $Q_{m_0}$ vanishes on $\U(n) \times \U(n)$ by lemma \ref{LEM:Homogeneity}. Consider $U,V$ unitary matrices such that $Q_{m_0}(U,V) = 0$.

There are now two possibilities. If for all $m \in S_f$, $Q_m(U,V) = 0$, then $\Psi_{U,V}(f) = 0$ so we are in the trivial case where $\sA(f) = -\strop(\V(f)) = \R^n$, which proves the inclusion. Else, there exists an $m_1 \in S_f$ such that $Q_{m_1}(U,V) \neq 0$. We will use the characterisation of classical tropical variety that uses the tropical polynomial. Given matrices $P(t),Q(t)$ that are invertible in $\RR$, we have
\begin{align*}
    \Psi_{P(t),Q(t)}(f)(z(t)) & = \sum_{m \in S_f} Q_m(P(t),Q(t))z^m(t)\\
    \Rightarrow \trop(\Psi_{P(t),Q(t)}(f))(y) & = \min_{m \in S_f}\{\val(Q_m(P(t),Q(t))) + m \cdot y\}.
\end{align*}
We want to find some matrices $(P(t),Q(t)) \in \GL_n(\RR)^2$ that verify the two following conditions, so we can use the lemma \ref{LEM:Tropical shift},
\begin{enumerate}
    \item $0 < \val(Q_{m_0}(P(t),Q(t))) < +\infty$,
    \item $\val(Q_{m_1}(P(t),Q(t))) = 0$.
\end{enumerate}
Assume that those two conditions are verified and let us show that $x \in \strop(\V(f))$. $P(t)$ and $Q(t)$ are invertible in $\RR$ so all the $Q_m(P(t),Q(t))$ have a non negative valuation. Let for all $s > 0$ and for all $m \in S_f$, the tropical monomials $M_m(s,y) = s\val(Q_m(P(t),Q(t))) + m \cdot y$ which equal $\val(Q_m(P(t^s),Q(t^s))) + m \cdot y$ when $s$ is rational. We have
\[
T_s(x) = \min_{m \in S_f}\{M_m(s,x)\} \tend{s}{0^+} \min_{\underset{Q_m(P(t),Q(t)) \neq 0}{m \in S_f}}\{m \cdot x\} = m_0 \cdot x < m \cdot x \textrm{ if } m \in S_f\backslash\{m_0\},
\]
so for some $a > 0$ small enough, $T_a(x) = M_{m_0}(a,x)$, and
\begin{align*}
    M_{m_0}(s,x) - M_{m_1}(s,x) & = s\val(Q_{m_0}(P(t),Q(t))) + m_0 \cdot x - s\val(Q_{m_1}(P(t),Q(t))) - m_1 \cdot x\\
    & = s\val(Q_{m_0}(P(t),Q(t))) + (m_0 - m_1) \cdot x\\
    & \tend{s}{+\infty} +\infty,
\end{align*}
so for $b > 0$ large enough, $M_{m_1}(b,x) < M_{m_0}(b,x)$ thus $T_b(x) \neq M_{m_1}(b,x)$. The hypothesis of the lemma \ref{LEM:Tropical shift} are verified thus according to this lemma and the proposition \ref{PRO:Union trop},
\[
x \in \overline{\bigcup_{s \in [a,b] \cap \Q} \V(\trop(\Psi_{P(t^s),Q^{-1}(t^s)}(f)))} \subset \strop(\V(f)).
\]
Now, all we have to do is to find matrices $P(t)$ and $Q(t)$ that verifies such conditions. Recall that there are matrices $U,V$ such that $Q_{m_1}(U,V) \neq Q_{m_0}(U,V) = 0$. As $U$ and $V$ are invertible, any matrix of the form $U + tA$ or $V + tA$ with $A \in \Mat_n(\C)$ are in $\GL_n(\RR)$. $Q_{m_0} \neq 0$ so there exists invertible matrices $A$ and $B$ such that $Q_{m_0}(A,B) \neq 0$. We choose $P(t) = U + t(A - U)$ and $Q(t) = V + t(B - V)$.

Condition 1 : Recall first of all that for all $m$, $\val(Q_m(P(t),Q(t))) \geq 0$. $Q_{m_0}(P(0),Q(0)) = Q_{m_0}(U,V) = 0$ by definition of $U$ and $V$ and $Q_{m_0}(P(1),Q(1)) = Q_{m_0}(A,B) \neq 0$ so $\val(Q_{m_0}(P(t),Q(t)))$ is neither 0, neither $+\infty$.

Condition 2 : it is a direct consequence of the fact that $Q_{m_1}(P(0),Q(0)) = Q_{m_1}(U,V) \neq 0$.

We found $P(t),Q(t)$ that satisfy the two wanted conditions, thus the proposition is proven.
\end{proof}
\begin{theorem}\label{THE:Bergman hypersurface}
    When $f \in \C[\GL_n]$, if $Y = \V(f)$,
    \[
    \lim_{\rho \rightarrow 0^+} \rho\sA(Y) = -\strop(Y) = \R^n\backslash(C^- \cup C^+).
    \]
\end{theorem}
\begin{proof}
This is a consequence of corollary \ref{COR:Inclusion1} and propositions \ref{PRO:Inclusion2}, \ref{PRO:Inclusion3, case 1} and \ref{PRO:Inclusion3, case 2}.
\end{proof}
Notice that the $C^-$ and $C^+$ only depend on the Newton polytope of $f$. Therefore, its tropical variety too.
\begin{example}
    Let $n = 2$ and $f$ any polynomial such that $\sNewt(f) = \Conv\{(-1,-1),(1,-2),(-2,1),(4,0),(0,4),(4,4)\}$ like on figure \ref{FIG:Example trop}. For example, $\d f(A) = \frac{1}{\det(A)} + \i\frac{\tr(A)^3}{\det(A)^2} - a_{11}a_{12}a_{21}a_{22} + 2\det(A)^4$.
\end{example}
Here, $N^-(1,1) = (-1,-1)$ and $N^+(1,1) = (4,4)$ are vertices of $\Newt(f)$. $C^-$ (resp. $C^+$) are the interiors of the normal cone of $N^-\mathbf{1}$ (resp. $N^+\mathbf{1}$), cf figure \ref{FIG:Example trop}.
\begin{figure}[H]
    \centering
    \includegraphics[width = 0.3\textwidth]{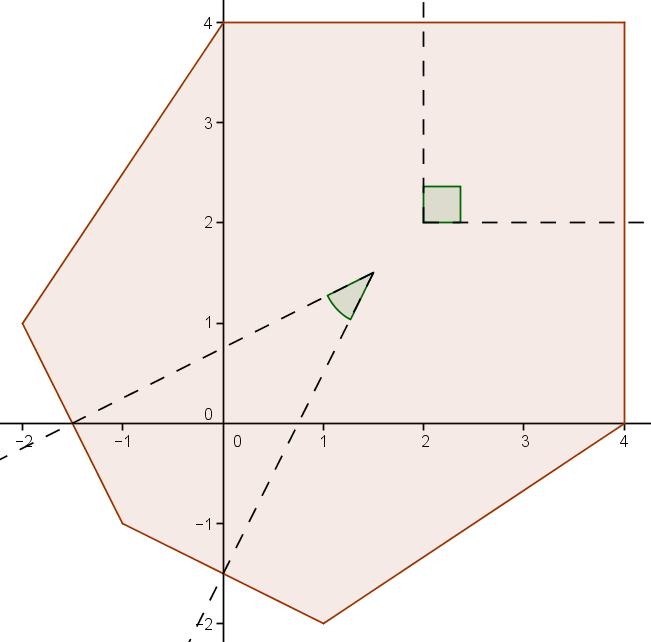}
    \includegraphics[width = 0.3\textwidth]{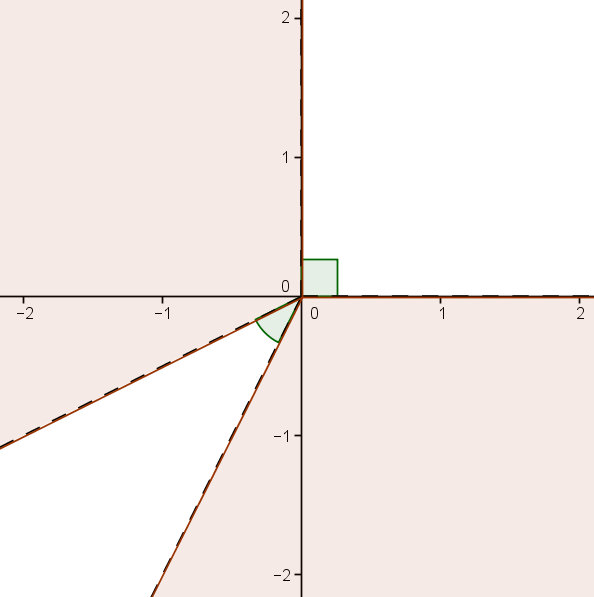} 
    \caption{The tropical variety (right) of a polynomial with a given Newton polytope (left).}
    \label{FIG:Example trop}
\end{figure}


\appendix

\section{Notations}\label{SEC:Notations}

$\Mat_n(R)$ is the set of $n \times n$ matrices with coefficients in the ring $R$.\\
$\GL_n(R)$ is the set of invertible ones.\\
$\C[X^\pm]$ is the set of Laurent polynomials with $n$ indeterminates (when $n$ is not ambiguous).\\
$[\![a,b]\!]$ is the set of relative integers between $a$ and $b$.\\
$\mathcal{S}_n$ is the group of permutations of $\{1, \ldots, n\}$.\\
$K[Y]$ is the ring of regular functions on the algebraic variety $Y$ over $K$.\\
$\chi_A$ is the characteristic polynomial of the matrix $A$.\\
$\Gamma(E)$ is the set of connected components of a topological space $E$.\\
$\Conv(S)$ is the convex hull of a subset $S$ of $\R^n$.\\
When $\d f = \sum_{m \in \Z^n} f_mX^m \in \C[X^\pm]$, $S_f = \{m \in \Z^n|f_m \neq 0\}$ is the support of $f$.\\
$\Newt(f) = \Conv(S_f)$ is its Newton polytope.\\
When $\d f(A\diag(z)B^{-1}) = \sum_{m \in \Z^n} Q_m(A,B)z^m$, $S_f = \{m \in \Z^n|Q_m \neq 0\}$ is the support of $f$.\\
$\sNewt(f) = \Conv(S_f)$ is its Newton polytope.\\
$\mathbf{1} = (1,\ldots,1) \in \R^n$.\\
When $K \subset \{1, \ldots, n\}$, $\ind_K \in \R^n$ is the vector such that for all $i$, $(\ind_K)_i = \ind_K(i)$.\\
When $1 \leq k \leq n$, $\varpi_k = \ind_{[\![k + 1 - n,n]\!]} = (0,\ldots,0,1,\ldots,1)$ with $k$ ones.\\
$B = \{(U,L) \in \GL_n(\C)^2|U \textrm{ upper triangular and } L \textrm{ lower triangular}\}$ is the Borel subgroup of $\GL_n(\C)^2$.\\
$\Lambda = \{\lambda \in \R^n|\lambda_1 \leq \cdots \leq \lambda_n\}$ is the set of dominant $B$-weights of $\GL_n(\C)$.\\
When $\lambda \in \Lambda$, $C(\lambda) = \Conv(\mathcal{S}_n \cdot \lambda)$.\\
$\Delta(f)$ is the moment polytope of $f$ (cf Section \ref{SEC:Representations}).\\
$V_\lambda$ is the irreducible $\GL_n(\C)$-representation of highest weight $\lambda \in \Lambda$.\\
$W_\lambda \cong V_\lambda \otimes V_\lambda^*$ is the irreducible $\GL_n(\C)^2$-representation of highest weight $\lambda \in \Lambda$.\\
$\overline{\mathcal{K}}$ is the field of complex Puiseux series.\\
$\d \val : \sum_{q \in \Q} a_qt^q \mapsto \min\{q \in \Q|a_q \neq 0\}$ is its natural valuation. $\val$ can also designate any valuation depending on the context.\\
$\RR$ is the ring of complex Puiseux series with a non negative valuation.\\
$\sval : A(t) \mapsto$ the valuations of the invariant factors of $A(t)$, which is the spherical valuation of a Puiseux series matrix.\\
$\Gamma_\val$ is the image of the valuation $\val$, it is a subgroup of $\R$.\\
$\V = \{x \in \R^n|x_1 \geq \cdots \geq x_n\}$ is the Weyl chamber of $\GL_n(\C)$.\\
$\strop(Y) = \overline{\{\sval(A(t))|A(t) \in Y(\overline{\mathcal{K}})\}}$ is the spherical tropical variety of a spherical variety $Y$.\\
$N^-$ (resp. $N^+$) is the lowest (resp. largest) integer $N$ such that $N\mathbf{1} \in \sNewt(f)$ when it is well-defined.\\
$C_F(\Delta)$ is the normal cone of the face $F$ of the convex polytope $\Delta$.\\
$C^-$ (resp. $C^+$) is the interior of $C_{\{N^-\mathbf{1}\}}(\sNewt(f))$ (resp. $C_{\{N^+\mathbf{1}\}}(\sNewt(f))$), or the empty set when $N^-$ (resp. $N^+$) does not exist.

\section{Some technical lemmas}\label{SEC:Lemmas}

The purpose of this section is to show useful but technical lemmas concerning measure theory, differential manifolds and complex analysis. The first one we need is the following,
\begin{appendix lemma}\label{LEM:Zero negligible}
    Let $P : \Mat_n(\C) \rightarrow \C$ be a nonzero polynomial with $n^2$ variables. Then the set of unitary zeros of $P$, i.e. $Z = P^{-1}\{0\} \cap \U(n) \subset \U(n)$, has measure zero, with respect to the natural probability Haar measure $\mu$ on $\U(n)$ (equivalently, any non-measure zero subset of $\U(n)$ is dense in $\GL_n(\C)$ for the Zariski topology).
\end{appendix lemma}
First of all,
\begin{appendix proposition}
    If $F : \R^n \rightarrow \R$ is $\mathcal{C}^1$ and such that $F^{-1}\{0\} \subset \R^n$ is non-measure zero for the Lebesgue measure, $\diff F^{-1}\{0\}$ is non-measure zero too.
\end{appendix proposition}
\begin{proof}
We prove the claim by contradiction. Let $N = \diff F^{-1}\{0\}$ that we assume to be measure zero. For all $x \in N^C$, let $1 \leq k_x \leq n$ and $U_x$ an open neighborhood of $x$ such that $\partial_{k_x}F(y) \neq 0$ for $y \in U_x$. Without loss of generality, we can assume that the $U_x$ are cubes centered around $x$. As $\R^n$ is a countable union of compact sets, we can find a countable family $S \subset \R^n$ such that $\d \R^n = \bigcup_{x \in S} U_x$. Therefore, if we show that for all $x \in N^C$, $\Vol(F^{-1}\{0\} \cap U_x) = 0$, it proves that $\Vol(F^{-1}\{0\}) = 0$. Indeed, for all $x$,
\begin{align*}
    \Vol(F^{-1}\{0\} \cap U_x) & = \int_{U_x} \ind_{\{0\}}(F(y)) \, \diff y\\
    & = \int_{x_1 - a}^{x_1 + a} \cdots \int_{x_n - a}^{x_n + a} \ind_{\{0\}}(F(y_1,\ldots,y_n)) \, \diff y_n \cdots \diff y_1,
\end{align*}
where $\d U_x = \prod_{k = 1}^n (x_k - a,x_k + a)$. As $\partial_{k_x}F$ does not vanish on $U_x$, for all $y_1,\ldots,y_{k_x - 1},y_{k_x + 1},\ldots,y_n$, the function $y_{k_x} \mapsto F(y)$ is strictly monotone. We deduce that it vanishes at most once. Therefore, for all $y_1,\ldots,y_{k_x - 1},y_{k_x + 1},\ldots,y_n$,
\[
\int_{x_{k_x} - a}^{x_{k_x} + a} \ind_{\{0\}}(F(y)) \diff y_{k_x} = 0
\]
so $\Vol(F^{-1}\{0\} \cap U_x) = 0$, which proves the proposition.
\end{proof}
\begin{appendix proposition}\label{PRO:Polynomial and manifold}
    The previous proposition remains true on a smooth Riemannian manifold $M$.
\end{appendix proposition}
\begin{proof}
As $M$ can be written as a countable union of open subsets that are diffeomorphic to bounded parts of $\R^n$ where $n$ is the dimension of $M$, we just apply the previous proposition on each of those subsets.
\end{proof}
Now, we can prove the lemma by induction on the degree $d$ of the polynomial $P$.
\begin{proof}\ \\
\noindent\underline{$d = 0$ :} If $P$ is constant non zero, $Z = P^{-1}\{0\} \cap \U(n) = \O$ is measure zero.\\

\noindent\underline{$d > 0$ :} Assume that $Z$ is non-measure zero for the natural measure $\mu$ of $\U(n)$. By the proposition \ref{PRO:Polynomial and manifold}, the set $Z' = \diff P^{-1}\{0\} \cap \U(n)$ is measure zero (the Haar measure for compact Lie groups is induced by its Riemannian structure when the group multiplications are isometries). Recall that at each point $U \in \U(n)$, $T_U\U(n) = \i\mathcal{H}(n)U$. Therefore, for all $U \in Z'$ and for all $A \in \mathcal{H}(n)$, $\diff P(U)(\i AU) = 0$. But $P$ is a polynomial, so it is in particular holomorphic. We deduce that $\diff P(u)$ is $\C$-linear so for all $A \in \mathcal{H}(n)$, $\diff P(U)(AU) = 0$. As $\Mat_n(\C) = \mathcal{H}(n) \oplus \i\mathcal{H}(n)$, we have in fact that $\diff P(U) = 0$ on all $\Mat_n(\C)$ when $U \in Z'$. Therefore, for each $1 \leq i,j \leq n$, $\partial_{X_{ij}}P(U) = \diff P(U)(E_{ij}) = 0$ when $U \in Z'$. But the degree of the $\partial_{X_{ij}}P$ is bounded by $d - 1$. We deduce by induction that they are all zero. It implies that $P$ is constant, which is absurd. This proves the lemma.
\end{proof}
\begin{remark}
Lemma \ref{LEM:Zero negligible} remains true when $P : \Mat_n(\C)^2 \to \C$.
Indeed, the lemma is equivalent to the following formula,
\[
\frac{1}{V_n}\int_{\U(n)} \ind_{\{0\}}(Q(U)) \, \diff\mu(U) = \ind_{\{0\}}(Q)
\]
where $Q : \Mat_n(\C) \rightarrow \C$ is a polynomial. So if we set for all $(A,B) \in \Mat_n(\C)^2$, $Q_A(B) = R_B(A) = P(A,B)$, the $Q_A$ and the $R_B$ are polynomials and
\begin{align*}
    \frac{1}{V_n^2}\int_{\U(n) \times \U(n)} \ind_{\{0\}}(P(U,V)) \, \diff\mu^2(U,V) & = \frac{1}{V_n^2}\int_{\U(n)}\int_{\U(n)} \ind_{\{0\}}(Q_V(U)) \, \diff\mu(U)\diff\mu(V)\\
    & = \frac{1}{V_n}\int_{\U(n)} \ind_{\{0\}}(Q_V) \, \diff\mu(V)\\
    & \leq \frac{1}{V_n}\int_{\U(n)} \ind_{\{0\}}(Q_V(A)) \, \diff\mu(V) \textrm{ for all fixed matrix $A$.}\\
    & = \frac{1}{V_n}\int_{\U(n)} \ind_{\{0\}}(R_A(V)) \, \diff\mu(V)\\
    & = \ind_{\{0\}}(R_A).
\end{align*}
If $P^{-1}\{0\}$ is non-measure zero, for all matrices $A$, the $R_A$ are zero thus $P = 0$.
\end{remark}
\begin{appendix lemma}
    If $f : \C^* \rightarrow \C^*$ is a continuous function such that there are non zero complex numbers $\alpha$ and $\beta$ and integers $(m_1,m_2) \in \Z^2$ verifying for all $\theta \in \R$, $f(r\e^{\i\theta}) \eq{r}{0} \alpha r^{-m_2}\e^{\i\theta m_2}$ and $f(r\e^{\i\theta}) \eq{r}{+\infty} \beta r^{m_1}\e^{\i\theta m_1}$, then $m_1 = m_2$.
\end{appendix lemma}
\begin{proof}
Under those hypothesis, for some $r_1,r_2 > 0$, we have
\begin{align*}
    r \leq r_1 & \Rightarrow \abs{f(r\e^{\i\theta}) - \alpha r^{-m_1}\e^{\i\theta m_1}} \leq \frac{1}{2}\abs{\alpha}r^{-m_1}\\
    r \geq r_2 & \Rightarrow \abs{f(r\e^{\i\theta}) - \beta r^{-m_2}\e^{\i\theta m_2}} \leq \frac{1}{2}\abs{\beta}r^{-m_2}\\
\end{align*}
As $\S^1$ is compact, we can choose $r_1$ and $r_2$ uniformly regarding to $\theta$. Up to increasing $r_2$ or decreasing $r_1$, we can assume that $\abs{\alpha}r_1^{-m_1} = \abs{\beta}r_2^{-m_2}$. We denote by $\rho$ this quantity. Consider now the following paths from $[0,1]$ to $\C^*$ for some $\theta \in \R$,
\begin{align*}
    \gamma_\theta^1 : t \mapsto \rho\exp(\i((1 - t)(\theta m_1 + \arg(\alpha)) + t(\theta m_2 + \arg(\beta)))) & \Rightarrow \gamma_\theta^1(0) = \alpha r_1^{-m_1}\e^{\i\theta m_1}, \gamma_\theta^1(1) = \beta r_2^{m_2}\e^{\i\theta m_2}\\\\
    \gamma_\theta^2 : t \mapsto (1 - t)\beta r_2^{-m_2}\e^{\i\theta m_2} + tf(r_2\e^{\i\theta}) & \Rightarrow \gamma_\theta^2(0) = \beta r_2^{m_2}\e^{\i\theta m_2}, \gamma_\theta^2(1) = f(r_2\e^{\i\theta})\\\\
    \gamma_\theta^3 : t \mapsto f(((1 - t)r_2 + tr_1)\e^{\i\theta}) & \Rightarrow \gamma_\theta^3(0) = f(r_2\e^{\i\theta}), \gamma_\theta^3(1) = f(r_1\e^{\i\theta})\\\\
    \gamma_\theta^4 : t \mapsto (1 - t)f(r_1\e^{\i\theta}) + t\alpha r_1^{m_1}\e^{\i\theta m_1} & \Rightarrow \gamma_\theta^4(0) = f(r_1\e^{\i\theta}), \gamma_\theta^4(1) = \alpha r_1^{m_1}\e^{\i\theta m_1}.
\end{align*}
\begin{figure}[H]
    \centering
    \includegraphics[width = 0.5\textwidth]{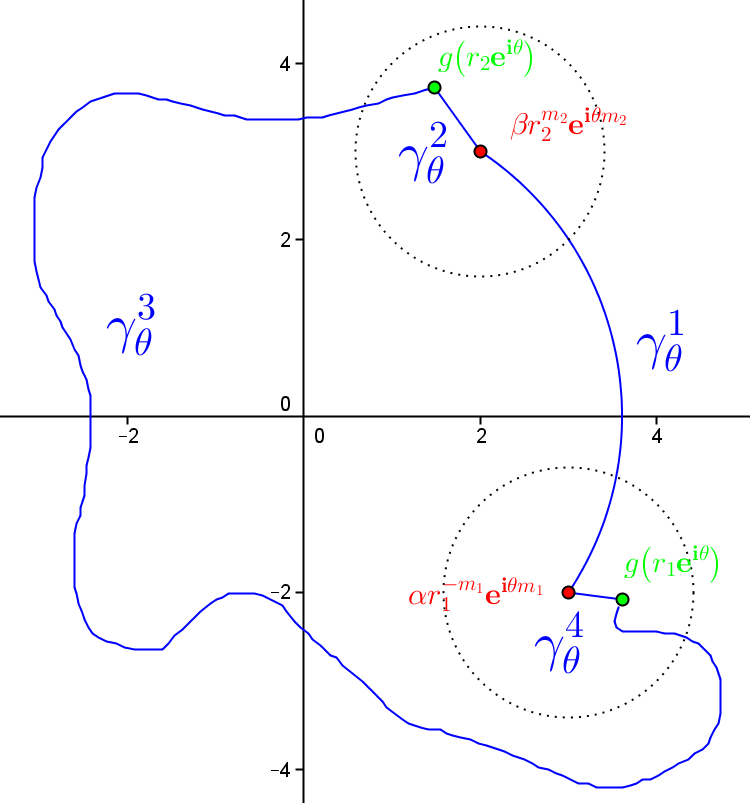}
    \caption{Illustration of the paths $\gamma_\theta^k$.}
\end{figure}

Let $\gamma_\theta = (\gamma_\theta^1 * \gamma_\theta^2)*(\gamma_\theta^3 * \gamma_\theta^4)$, which is a path from $[0,1]$ to $\C^*$ because $f$ does not vanish. Notice that $\theta \mapsto \gamma_\theta$ is a homotopy so all the $\gamma_\theta$ are homotopic the ones with the others. For any $k \in \{2,3,4\}$, $\theta \mapsto \gamma_\theta^k$ is $2\pi$-periodic so
\[
\gamma_0 \sim \gamma_{2\pi} \Rightarrow \gamma_0^1 \sim \gamma_{2\pi}^1 \Rightarrow \mathrm{Ind}_0\left(\gamma_{2\pi}^1 * \overline{\gamma_0^1}\right) = 0
\]
And after noticing that for all $t$, $\gamma_{2\pi}^1(t) = \e^{2\i\pi(m_2 - m_1)t}\gamma_0^1(t)$ we can compute that
\begin{align*}
    \mathrm{Ind}_0\left(\gamma_{2\pi}^1 * \overline{\gamma_0^1}\right) & = \int_0^1 \frac{\gamma_{2\pi}^1{}'(t)}{\gamma_{2\pi}^1(t)} \frac{\diff t}{2\i\pi} - \int_0^1 \frac{\gamma_0^1{}'(t)}{\gamma_0^1(t)} \frac{\diff t}{2\i\pi}\\
    & = \int_0^1 \frac{2\i\pi(m_2 - m_1)\e^{2\i\pi(m_2 - m_1)t}\gamma_0^1(t) + \e^{2\i\pi(m_2 - m_1)t}\gamma_0^1{}'(t)}{\e^{2\i\pi(m_2 - m_1)t}\gamma_0^1(t)} \frac{\diff t}{2\i\pi} - \int_0^1 \frac{\gamma_0^1{}'(t)}{\gamma_0^1(t)} \frac{\diff t}{2\i\pi}\\
    & = m_2 - m_1.
\end{align*}
As this quantity is null, we deduce that $m_1 = m_2$.
\end{proof}

\begin{appendix lemma}\label{LEM:Homogeneity}
    When $h : \Mat_n(\C) \rightarrow \C$ is a homogeneous polynomial with homogeneity coefficient $m \in \N^n$ regarding to the columns in the sens that for all $A$, $\lambda$, $h(A\diag(\lambda)) = \lambda^mh(A)$ and $h$ does not vanish on $\U(n)$, then $m_1 = \cdots = m_n$.
\end{appendix lemma}
\begin{proof}
\underline{First case, $n = 1$ :} Trivial.\\

\underline{Second case, $n = 2$ :} Let $f : \fonction{\C^2}{\C}{(x,y)}{h\begin{pmatrix} x & y \\ 1 & 1 \end{pmatrix}}$. By homogeneity and because $h$ does not vanish on $\U(2)$, $h$ does not vanish on any invertible matrix $A$ if its columns are orthonormal so for any $z \in \C^*$,
\[
f\left(z,-\frac{z}{\abs{z}^2}\right) = h\begin{pmatrix} z & -\frac{z}{\abs{z}^2} \\ 1 & 1 \end{pmatrix} \neq 0.
\]
Moreover, $h$ is a polynomial so $f$ is and
\[
f(x,y) = h\begin{pmatrix} x & y \\ 1 & 1 \end{pmatrix} = x^{m_1}y^{m_2}h\begin{pmatrix} 1 & 1 \\ \frac{1}{x} & \frac{1}{y} \end{pmatrix} \landau{\abs{x},\abs{y}}{+\infty} \mathrm{O}(x^{m_1}y^{m_2}).
\]
We deduce that $f$ can be written as
\[
f(x,y) = \sum_{p_1 = 0}^{m_1}\sum_{p_2 = 0}^{m_2} f_{p_1,p_2}x^{p_1}y^{p_2}.
\]
Now, notice that
\begin{align*}
    f\left(z,-\frac{z}{\abs{z}^2}\right) & = h\begin{pmatrix} z & -\frac{z}{\abs{z}^2} \\ 1 & 1 \end{pmatrix}\\
    & = \left(\frac{z}{\abs{z}^2}\right)^{m_2}h\begin{pmatrix} z & -1 \\ 1 & \overline{z} \end{pmatrix}\\
    & \landau{z}{0} \frac{1}{\overline{z}^{m_2}}h\begin{pmatrix} 0 & -1 \\ 1 & 0 \end{pmatrix} + \circ\left(\frac{1}{\abs{z}^{m_2}}\right).
\end{align*}
and
\[
f\left(z,-\frac{z}{\abs{z}^2}\right) = \sum_{p_1 = 0}^{m_1}\sum_{p_2 = 0}^{m_2} f_{p_1,p_2}z^{p_1}\frac{1}{\overline{z}^{p_2}} \landau{z}{0} f_{0,m_2}\frac{(-1)^{m_2}}{\overline{z}^{m_2}} + \circ\left(\frac{1}{\abs{z}^{m_2}}\right).
\]
so $\d f_{0,m_2} = (-1)^{m_2}h\begin{pmatrix} 0 & -1 \\ 1 & 0 \end{pmatrix} \neq 0$ because this matrix is unitary. With the same kind or argument, we could show that $f_{m_1,0} \neq 0$. Let $z = r\e^{\i\theta} \in \C^*$.
\begin{align*}
    f\left(z,-\frac{z}{\abs{z}^2}\right) & = f\left(r\e^{\i\theta},-\frac{\e^{\i\theta}}{r}\right)\\
    & = \sum_{p_1 = 0}^{m_1}\sum_{p_2 = 0}^{m_2} (-1)^{p_2}f_{p_1,p_2}r^{p_1 - p_2}\e^{\i\theta(p_1 + p_2)}\\
    & \eq{r}{0} (-1)^{m_2}f_{0,p_2}r^{-m_2}\e^{\i\theta m_2}\\
    \mathrm{and} & \eq{r}{+\infty} f_{p_1,0}r^{m_1}\e^{\i\theta m_1} \textrm{ using the same argument.}
\end{align*}
Using the previous lemma, we obtain that $m_1 = m_2$.\\

\underline{Third case : $n \geq 3$} Let $1 \leq i \leq n - 1$ $\tilde{h} : \fonction{\Mat_2(\C)}{\C}{A}{h\begin{pmatrix} I_{i - 1} & 0 & 0 \\ 0 & A & 0 \\ 0 & 0 & I_{n - i - 1} \end{pmatrix}}$. $\tilde{h}$ verify the hypothesis of the lemma with $m_i,m_{i + 1}$. Using the case $n = 2$, we deduce that $m_i = m_{i + 1}$. It is true for all $i$ so $m_1 = \cdots = m_n$.
\end{proof}


\bibliographystyle{alpha}
\bibliography{bibliography.bib}

\end{document}